\documentclass[a4paper,leqno,10pt]{amsart}

\usepackage{epsf,graphicx,epsfig}
\usepackage{amscd}
\usepackage{amsmath,latexsym,amssymb,amsthm}
\usepackage[nospace,noadjust]{cite}
\usepackage{textcomp}
\usepackage{setspace,cite}
\usepackage{lscape,fancyhdr,fancybox}
\usepackage[all,cmtip]{xy}
\usepackage{tikz}
\usetikzlibrary{shapes,arrows,decorations.markings}
\usepackage{graphicx}

\usepackage{color}
\usepackage{url}
\usepackage{enumerate}
\usepackage[mathscr]{euscript}

  \theoremstyle{plain}

\swapnumbers
    \newtheorem{thm}{Theorem}[section]
    \newtheorem{prop}[thm]{Proposition}
   
    \newtheorem{corollary}[thm]{Corollary}
    
    \newtheorem{subsec}[thm]{}
\theoremstyle{definition}
    \newtheorem{defn}[thm]{Definition}
    \newtheorem{exam}[thm]{Example}

\theoremstyle{remark}
     \newtheorem{remark}[thm]{Remark}

\newcommand{\largewedge}{\mbox{\Large $\wedge$}}

\title{}
\author{}
\date{}
\usepackage{amssymb}

\usepackage{hyperref}
\hypersetup{
	colorlinks,
	citecolor=blue,
	filecolor=black,
	linkcolor=blue,
	urlcolor=black
}

\begin{document}
\title{Poisson-Nijenhuis groupoids}

\author{Apurba Das}
\email{apurbadas348@gmail.com}
\address{Stat-Math Unit,
Indian Statistical Institute, Kolkata 700108,
West Bengal, India.}


\begin{abstract}
We define multiplicative Poisson-Nijenhuis structures on a Lie groupoid which
extends the notion of symplectic-Nijenhuis groupoid introduced by Sti$\acute{\text{e}}$non and Xu.
We also introduce a special class of Lie bialgebroid structure on a Lie algebroid $A$, called P-N Lie bialgebroid, which defines a hierarchy of compatible Lie 
bialgebroid structures on $A$. 
We show that under some topological assumption
on the groupoid,
there is a one-to-one correspondence between multiplicative Poisson-Nijenhuis structures
on a Lie groupoid and P-N Lie bialgebroid structures on the corresponding Lie algebroid.\\
\end{abstract}

\noindent
\thispagestyle{empty}

\maketitle

\tableofcontents

\section{Introduction}
Poisson geometry originated in the last century in the hamiltonian
formulation of classical mechanics and became formalized in the language of modern differential geometry by introducing Poisson manifolds and Poisson brackets.
A Poisson manifold is a 
manifold equipped with a bivector field whose Schouten bracket with itself vanishes. 
Among the various examples of Poisson manifolds, Poisson Lie groups \cite{lu-wein} are of particular interest. The more general notion of Poisson groupoid was introduced by 
Weinstein \cite{wein} as a unification of both Poisson Lie group and symplectic groupoid \cite{wein4}. 
In order to find out the infinitesimal ingredient of a Poisson groupoid, one observes that the dual bundle $A^*$ of its Lie algebroid $A$
also carries a Lie algebroid structure. Moreover, these two Lie
algebroid structures on duality satisfy the criteria of a Lie bialgebroid in the sense of
Mackenzie and Xu \cite{mac-xu}. In \cite{mac-xu2}, the same authors showed that 
under certain topological assumption, any Lie bialgebroid also integrates to a Poisson groupoid.

\medskip

The notion of Poisson-Nijenhuis structure has been studied by Magri and 
Morosi \cite{mag-mor} in the theory of completely integrable Hamiltonian systems.
A Poisson-Nijenhuis manifold is a manifold equipped with a Poisson bivector field 
and a Nijenhuis tensor which are compatible in a certain way that it induces a hierarchy of compatible Poisson
structures \cite{yks-mag, vais2}. A symplectic-Nijenhuis manifold is a Poisson-Nijenhuis manifold with non-degenerate Poisson structure. A Poisson-Nijenhuis manifold is said to be integrable if its Poisson structure is integrable. 
In \cite{stienon-xu}, Sti\'enon and Xu
introduced a notion of symplectic-Nijenhuis groupoid as a global object corresponding to an integrable Poisson-Nijenhuis manifold.
A symplectic-Nijenhuis groupoid is a Lie groupoid equipped with a symplectic-Nijenhuis structure such that the symplectic form 
and the Nijenhuis tensor are both multiplicative.

\medskip

The aim of this paper is to study 
 multiplicative Poisson-Nijenhuis structures on a Lie groupoid,
or a Poisson-Nijenhuis groupoid. This would generalize  both Poisson groupoid and 
symplectic-Nijenhuis groupoid of Sti\'enon and Xu \cite{stienon-xu}. A Poisson-Nijenhuis groupoid
is a Lie groupoid with a Poisson-Nijenhuis structure such that the Poisson bivector field and the Nijenhuis tensor are both multiplicative. 
Any Poisson groupoid with a multiplicative complementary $2$-form, pair groupoid of Poisson-Nijenhuis manifolds, 
holomorphic Poisson groupoids \cite{jotz-stienon-xu} are also examples of Poisson-Nijenhuis groupoids (cf. Example \ref{pn-grpd-exam}).
We show that the base of a Poisson-Nijenhuis groupoid carries a Poisson-Nijenhuis structure such that the source map is a Poisson-Nijenhuis map (cf. Proposition \ref{pn-grpd-base-pn}).

\medskip

We also introduce a notion of P-N Lie bialgebroid structure on a Lie algebroid $A$. 
More precisely, a P-N Lie bialgebroid structure on a Lie algebroid $A$ consists of a Lie bialgebroid $(A, A^*)$ together with an infinitesimal
multiplicative Nijenhuis tensor $N_A : TA \rightarrow TA$ such that $(\pi_A, N_A)$ 
defines a Poisson-Nijenhuis structure on $A$, where
$\pi_A$ being the linear Poisson structure on $A$ induced from the Lie algebroid $A^*$ (cf. Definition \ref{inf-mul-pn-lie-bialgbd}). 
The Lie bialgebroid of a Poisson-Nijenhuis manifold, holomorphic Lie bialgebroids are examples of P-N Lie bialgebroids (cf. Example \ref{exam-pn-lie-bialgbd}).
We show that the base of a P-N Lie bialgebroid carries a natural Poisson-Nijenhuis structure (cf. Proposition \ref{pn-lie-bialgbd-base-pn}). 
Moreover, given a P-N Lie bialgebroid structure on $A$, there is a  
hierarchy of Lie bialgebroid 
structures on $A$ that are compatible in a certain sense (cf. Proposition \ref{pn-lie-bialgbd-comp-bialgbd}).

\medskip

Next, we show that the infinitesimal form of Poisson-Nijenhuis groupoids are P-N Lie bialgebroids (cf. Proposition \ref{pn-grpd-pn-lie-bialgbd}).
More generally, given a source-connected and source-simply connected
Lie groupoid $G \rightrightarrows M$ with Lie algebroid $A$, there is a one-to-one correspondence between multiplicative
Poisson-Nijenhuis structures on $G$ and P-N Lie bialgebroid
structures on $A$ (cf. Theorem \ref{1-1,pn-grpd-pn-lie-bialgbd}).

\medskip

We also introduce an important class of subgroupoids of a Poisson-Nijenhuis groupoid, called coisotropic-invariant subgroupoids and these subgroupoids corresponds to, so called
coisotropic-invariant subalgebroids of the corresponding P-N Lie bialgebroid. Finally, we introduce P-N action of a Poisson-Nijenhuis groupoid on a
Poisson-Nijenhuis manifold generalizing Poisson action of Liu, Weinstein and Xu \cite{liu-wein-xu}. 
It turns out that the moment map of a P-N action is a Poisson-Nijenhuis map.

\medskip

On the way to our result, we study Nijenhuis tensors on a manifold (or on a Lie groupoid) and compatible Lie algebroid structures on a vector bundle.
We derive several useful results which are used in the main contents of the paper.

\medskip

\noindent {\bf Organization.} The paper is organized as follows. In section \ref{sec2}, we recall basic definitions and some known results. Section \ref{sec3} concerns
about Nijenhuis tensors on a manifold/Lie groupoid, while in Section \ref{sec4}, we study compatible Lie algebroid structures on a vector bundle. Sections \ref{sec5}, \ref{sec6} and \ref{sec7} are devoted to the study of
Poisson-Nijenhuis groupoids, P-N Lie bialgebroids and their infinitesimal correspondences. In section \ref{sec8}, we deal with coisotropic-invariant subgroupoids of a
Poisson-Nijenhuis groupoid and their infinitesimal. Finally, in section \ref{sec9}, we define P-N action.

\medskip

\noindent {\bf Notation.}
Given a Lie groupoid $G \rightrightarrows M$, by $s, t : G \rightarrow M$, we denote the source map and the target map.
Two elements $g, h \in G$ are composable if $t(g) = s (h).$ By $i : G \rightarrow G , ~ g \mapsto g^{-1}$, we denote the inversion and $\epsilon : M \rightarrow G, ~ x \mapsto \epsilon(x)$
the unit map.
The orbit set of a Lie groupoid $G \rightrightarrows M$ is the quotient $M / \sim$ determined by the equivalence relation `$\sim$' on $M$ : two points $x, y \in M$ are $\sim$ equivalent, if there
exists an element $g \in G$ such that $s(g) = x,~ t(g) = y$.

For any Lie groupoid $G \rightrightarrows M$, the tangent bundle  $TG$ carries a Lie groupoid structure over $TM$ whose structure maps are the tangent prolongation of the structure
maps of $G$. This Lie groupoid is called the tangent Lie groupoid of $G \rightrightarrows M$. Similarly, the cotangent bundle $T^*G$ has a Lie groupoid
structure but less obviously \cite{mac-xu2}. The base of this groupoid is $A^*$, where $A$ is the Lie algebroid of the grouopid. The unit map of this groupoid identifies
$A^*$ with the conormal bundle $(TM)^0 \hookrightarrow T^*G$ of the submanifold $M$ of $G$.

Given a Lie algebroid $A \rightarrow M$, the bundle $TA \rightarrow TM$ carries a natural Lie algebroid structure, called the tangent Lie algebroid of $A$.
Moreover, the cotangent bundle $T^*A$ also carries a Lie algebroid structure over $A^*$ (see \cite{mac-xu} for more details).

\section{Preliminaries}\label{sec2}
Let $M$ be a smooth manifold and $N : TM \rightarrow TM$ be a vector valued $1$-form, or a $(1,1)$-tensor on $M$. Then its Nijenhuis torsion $\tau_N$ is a vector valued $2$-form defined by
$$\tau_N(X, Y) := [NX, NY] - N ( [NX, Y] + [X, NY] - N[X, Y]),~~ \text{ for }X, Y \in \Gamma(TM).$$
\begin{defn}
An $(1,1)$-tensor $N$ is called a {\sf Nijenhuis tensor} if its Nijenhuis torsion vanishes.
\end{defn}
Given a Nijenhuis tensor $N$, one can define a new Lie algebroid structure on $TM$ deformed by $N$. We denote this Lie algebroid by $(TM)_N$. The bracket $[~,~]_N$ and anchor $id_N$ of this deformed Lie algebroid
are given by
$$ [X, Y]_N = [NX, Y] + [X, NY] - N[X,Y] \text{~ and ~} id_N = N , ~~ \text {for all } X, Y \in \Gamma(TM).$$
Moreover, if $N$ is a Nijenhuis tensor on $M$, then for any $k \geq 1$, $N^k$ is a Nijenhuis tensor on $M$. Thus, we have a hierarchy of deformed Lie algebroid
structures $([~,~]_{N^k}, id_{N^k} = N^k)$ on $TM$ \cite{yks-mag}.

\begin{defn}
 Let $M$ be a smooth manifold. An $(1,1)$-tensor $N : TM \rightarrow TM$ is said to be tangent to a submanifold $S \hookrightarrow M$, or,
$S$ is called an {\sf invariant submanifold} of $M$ with respect to $N$ if $N (TS) \subseteq TS$. In that case, $N$
restricts to an $(1,1)$-tensor $N_S : TS \rightarrow TS$ on the submanifold $S$.
\end{defn}
If $N$ is a Nijenhuis tensor on $M$ tangent to $S \hookrightarrow M$, then $N_S$ is a Nijenhuis tensor on $S$ and the deformed Lie algebroid $(TS)_{N_S} \rightarrow S$ is a Lie
subalgebroid of $(TM)_N \rightarrow M$.

\medskip

Let $M$ be a smooth manifold and $\pi \in \Gamma{(\largewedge^2 TM)}$ a bivector field on $M$. Then one can define a skew-symmetric bracket $[~,~]_\pi$ on the
space $\Omega^1(M)$ of $1$-forms on $M$, given by
$$[\alpha, \beta]_\pi := \mathcal{L}_{\pi^\sharp \alpha} \beta - \mathcal{L}_{\pi^\sharp \beta} \alpha - d(\pi(\alpha, \beta)), ~~ \text{ for all } \alpha, \beta \in \Gamma(T^*M),$$
where $\pi^\sharp : T^*M \rightarrow TM,~ \alpha \mapsto \pi(\alpha, -)$ is the bundle map induced by $\pi$.
If $\pi$ is a Poisson bivector (that is, $[\pi, \pi] = 0$), then the cotangent bundle $T^*M$ with the above bracket and the bundle map $\pi^\sharp$
forms a Lie algebroid \cite{mac-xu}. We call this Lie algebroid as the cotangent Lie algebroid of the Poisson manifold $(M, \pi)$ and is denoted by
$(T^*M)_\pi.$

\begin{defn}(Poisson-Nijenhuis manifolds)
 A {\sf Poisson-Nijenhuis manifold} is a manifold $M$ together with a Poisson bivector $\pi \in \Gamma{(\largewedge^2 TM)}$ and a Nijenhuis tensor $N$ such that they
 are compatible in the follwing senses:
\begin{itemize}
 \item[(i)] $ N \circ \pi^\sharp = \pi^\sharp \circ N^*$ \quad (thus, $N \circ \pi^{\sharp}$ defines a bivector field $N \pi$ on $M$),
 \item[(ii)] $C(\pi, N) \equiv 0 ,$
\end{itemize}
where
$$C(\pi, N) (\alpha, \beta) := [\alpha, \beta]_{N \pi} - ([N^*\alpha, \beta]_{\pi} + [\alpha, N^*\beta]_{\pi} - N^* [\alpha, \beta]_{\pi}), \text{~ for ~} \alpha, \beta \in \Omega^1(M).$$
The skew-symmetric $C^\infty(M)$-bilinear operation $C(\pi, N)(-,-)$ on the space of $1$-forms is called
the Magri-Morosi concominant of the Poisson structure $\pi$ and the Nijenhuis tensor $N$. 
\end{defn}
 
A Poisson-Nijenhuis manifold as above is denoted by the triple $(M, \pi, N)$. If $\pi$ is non-degenerate, that is, defines a symplectic
structure $\omega$ on $M$, then $(M, \omega, N)$ is said to be a symplectic-Nijenhuis manifold.

\begin{exam}\label{pn-manifold-exam}
(i) Any Poisson manifold $(M, \pi)$ is a Poisson-Nijenhuis manifold $(M, \pi, N = id)$.

\medskip

(ii) (complemented Poisson manifolds) 

Let $(M, \pi)$ be a Poisson manifold with the cotangent Lie algebroid $(T^*M)_\pi = (T^*M, [~,~]_{\pi} , \pi^\sharp)$.
We denote the extended Gerstenhaber bracket on $\Omega^\bullet (M) = \Gamma (\largewedge^\bullet T^*M)$ by the same symbol $[~,~]_\pi$.
A $2$-form $\omega \in \Omega^2(M)$ is called a complementary $2$-form if
$$ [\omega, \omega]_{\pi} = 0 .$$
A Poisson manifold together with a complementary $2$-form is called a {\sf complemented Poisson manifold}. Given a complemented Poisson manifold
$(M, \pi, \omega)$, if $\omega$ satisfies
$$ \iota_{\pi^\sharp \alpha} d\omega = 0, ~~ \text{ for all } \alpha \in \Omega^1(M),$$
(in particaular, if $d\omega = 0$) then $(\pi, N = \pi^\sharp \circ \omega^\sharp)$ defines a Poisson-Nijenhuis structure on $M$ \cite{vais2}.
Here $\omega^\sharp : TM \rightarrow T^*M,~ X \mapsto \iota_X \omega$ denotes the bundle map induced by $\omega$.

\medskip

(iii) (holomorphic Poisson manifolds) 

A {\sf holomorphic Poisson manifold} is a complex manifold $M$ whose sheaf of holomorphic functions $\mathcal{O}_M$
is a sheaf of Poisson algebras \cite{geng-stie-xu1}. Given a complex manifold $M$, there is a one-to-one correspondence between holomorphic Poisson structures on $M$ and holomorphic
bivector fields $\pi \in \Gamma(\largewedge^2T^{1,0}M)$ satisfying $[\pi, \pi] = 0.$ One can write any holomorphic bivector field $\pi$ as $\pi = \pi_R + i \pi_I$, where $\pi_R, \pi_I \in \Gamma(\largewedge^2TM)$
are smooth bivector fields on the underlying real manifold $M$ (by forgetting the complex structure $J : TM \rightarrow TM$). 
A holomorphic bivector field $\pi = \pi_R + i \pi_I$ defines a holomorphic Poisson structure on $M$
if and only if $(\pi_I, J)$ defines a Poisson-Nijenhuis structure on $M$ and $\pi_R^\sharp = J \circ \pi_I^\sharp$ \cite{geng-stie-xu1}.
\end{exam}

Let $(M, \pi, N)$ be a Poisson-Nijenhuis manifold. Then it is clear that for each $k \geq 0$, the maps
$$ N^k \circ \pi^\sharp : T^*M \rightarrow TM $$
are skew-symmetric and therefore define a hierarchy of bivector fields $\pi_k := N^k\pi$ by $\pi_k^\sharp = N^k \circ \pi^\sharp$. Moreover, it turns out that these bivectors $\pi_k$ are Poisson
bivectors and they are compatible \cite{yks-mag}. That is,
$$ [ \pi_k , \pi_k] = 0 ~~ \text{~~ and ~~} ~~ [\pi_k, \pi_l] = 0 , ~~ \text{for all } k, l \geq 0 .$$

\begin{defn}
 Let $(M, \pi, N)$ and $(\widetilde{M}, \widetilde{\pi}, \widetilde{N})$ be two Poisson-Nijenhuis manifolds. A 
smooth map $f : M \rightarrow \widetilde{M}$ is called a {\sf Poisson-Nijenhuis map} (or P-N map in short) if $f$ is
a Poisson map, that is, $f_* \pi = \widetilde{\pi}$ and commute with the Nijenhuis tensors, that is, 
$\widetilde{N} \circ f_* = f_* \circ N$.
\end{defn}

If $f$ is a P-N map, then $f$ preserves the corresponding hierarchy of Poisson structures. That is, if $\pi_k = N^k \pi$
and $\widetilde{\pi}_k = \widetilde{N}^k \widetilde{\pi}$ denote the corresponding hierarchy of Poisson bivectors, then
$f_* \pi_k = \widetilde{\pi}_k$, for all $k \geq 0$. An anti-P-N map is an anti-Poisson map that commute with the Nijenhuis tensors.

\begin{defn}\label{lie bialgebroid} \cite{mac-xu}
 A {\sf Lie bialgebroid} over $M$ is a pair $(A, A^*)$ of Lie algebroids over $M$ in duality, where the differential $d_{A^*}$ on $\Gamma (\largewedge^{\bullet} A)$ defined by the Lie algebroid structure on $A^*$ and the Gerstenhaber bracket on 
$\Gamma (\largewedge^{\bullet} A)$ defined by the Lie algebroid structure on $A$ satisfies
$$d_{A^*} [X, Y] = [d_{A^*} X , Y] + [X, d_{A^*}Y] ,~~ \text{ for all } X, Y \in \Gamma A.$$
\end{defn}

\begin{remark}\label{lie-bialgebroid-diff}
A Lie bialgebroid may also be defined as a pair $(A, \delta)$ of a Lie algebroid $A$ together with a degree one differential $\delta : \Gamma{(\largewedge^\bullet A)} \rightarrow \Gamma{(\largewedge^{\bullet +1} A )} $
on the corresponding Gerstenhaber algebra $(\Gamma{(\largewedge^\bullet A)} , \wedge, [~,~])$. Here $\delta$ is the differential $d_{A^*}$ of the Lie algebroid $A^*$.
\end{remark}

It is known that for any vector bundle $A \rightarrow M$, there is a canonical isomorphism $R : T^*A^* \rightarrow T^*A$ defined as follows \cite{mac-xu}. Note that, elements of $T^*A$
can be represented as $(w , X, \phi)$, where $w \in T^*_xM$, $X \in A_x$, $\phi \in A_x^*$ for some $x \in M$. In these terms, the map $R$ is defined by
 $R (w, \phi, X) = (- w, X, \phi)$. Then we have the following equivalent characterization of Lie bialgebroid (Theorem 6.2 \cite{mac-xu}).
\begin{thm}
 Let $A$ be a Lie algebroid over $M$ such that the dual bundle $A^*$ also carries a Lie algebroid structure. Then the pair $(A, A^*)$ is a Lie bialgebroid over
$M$ if and only if
\[
\xymatrixrowsep{0.4in}
\xymatrixcolsep{0.6in}
\xymatrix{
T^*A^* \ar[r]^{\pi_A^\sharp \circ R} \ar[d] & TA \ar[d] \\
A^*  \ar[r]_{\rho_*} & TM
}
\]
is a Lie algebroid morphism, where the domain $T^*A^* \rightarrow A^*$ is the cotangent Lie algebroid of $A^*$ induced from the Lie algebroid structure on $A$, the target $TA \rightarrow TM$
is the tangent Lie algebroid of $A$ and $\pi_A$ is the linear Poisson structure on $A$ induced from the Lie algebroid structure on $A^*$.
\end{thm}

If $(M, \pi, N)$ is a Poisson-Nijenhuis manifold, then $((TM)_N , (T^*M)_\pi)$ forms a Lie bialgebroid over $M$ \cite{yks2}. It follows from the duality of
a Lie bialgebroid that $((T^*M)_\pi, (TM)_N )$ is also a Lie bialgebroid.
Thus, in terms of the degree one differential, $((T^*M)_\pi, d_N)$ forms a Lie bialgebroid,
 where $d_N$ is the differential of the deformed tangent Lie algebroid $(TM)_N$. The characterization
of Lie bialgebroids arising from Poisson-Nijenhuis manifolds is given by the following \cite{stienon-xu}.

\begin{prop}\label{lie-bialgbd-pn}
 Let $(M, \pi)$ be a Poisson manifold. A Lie bialgebroid $((T^*M)_{\pi}, \delta)$ is induced by a Poisson-Nijenhuis structure if and only if
$[\delta, d] = 0$, where $d$ is the usual de Rham differential.
\end{prop}

Poisson groupoids \cite{wein} were introduced by Weinstein as a unification of both Poisson Lie group and symplectic groupoid.
\begin{defn}\label{pois-grpd}
 A {\sf Poisson groupoid} is a Lie groupoid $G \rightrightarrows M$ equipped with a Poisson structure $\pi$ on $G$ such that
the graph 
$$\text{Gr}(m) = \{(g,h, g h)| ~ t(g) = s(h)\}$$
of the groupoid multiplication map is a coisotropic submanifold of the product Poisson manifold $G \times G \times \overline{G}$. Here $\overline{G}$ denotes the manifold $G$ with the
opposite Poisson structure $- \pi$. A Poisson groupoid as above is denoted by $(G \rightrightarrows M, \pi)$.
\end{defn}

The above condition is equivalent to the fact that $\pi^\sharp : T^*G \rightarrow TG$ is a Lie groupoid morphism from the cotangent Lie groupoid $T^*G \rightrightarrows A^*$ to
the tangent Lie groupoid $TG \rightrightarrows TM$.

Given a Lie groupoid $G \rightrightarrows M$ with Lie algebroid $A$,
it is known that there is a canonical isomorphism $$ j_G : TA = T (Lie~ G) \rightarrow Lie (TG)$$
between corresponding double vector bundle structures on them. The map $j_G$ is the restriction of the canonical involution map $\sigma : T(TG) \rightarrow T(TG)$.
Similarly, one defines an isomorphism
$$ j_G^{'} : Lie (T^*G) \rightarrow T^* (Lie ~ G) = T^*A $$
between corresponding double vector bundle structures on them. The maps $j_G$ and $j_G^{'}$ are related by
$j_G^{'} = j_G^{*} \circ i_G$, where $i_G : Lie (T^*G) \rightarrow (Lie (TG))^*$ is the canonical double vector bundle isomorphism induced from the
canonical pairing $T^*G \oplus_G TG \rightarrow \mathbb{R}$ (see \cite{mac-xu, mac-xu2} for more details).

Let $(G \rightrightarrows M, \pi)$ be a Poisson groupoid with Lie algebroid $A$. It turns out that the dual bundle $A^*$ also carries a Lie algebroid structure
and the pair $(A, A^*)$ forms a Lie bialgebroid \cite{mac-xu}.
 Conversely, if $A$ is the Lie algebroid of a $s$-connected and $s$-simply connected Lie grouopid
$G \rightrightarrows M$, then any Lie bialgebroid $(A, A^*)$ integrates to a unique Poisson structure $\pi$ on $G$ making $(G \rightrightarrows M , \pi)$
into a Poisson groupoid with given Lie bialgebroid \cite{mac-xu2}.
If $\pi_A$ being the linear Poisson structure on $A$ induced from the Lie algebroid $A^*$, then we have
$$ \pi_A^\sharp = j_G^{-1} \circ Lie (\pi^\sharp) \circ j_G^{' -1} : T^*A \rightarrow TA.$$
The Lie algebroid differential $d_{A^*} : \Gamma (\largewedge^\bullet A) \rightarrow  \Gamma (\largewedge^{\bullet +1} A)$ of $A^*$ is also denoted by $\delta_\pi$
and is the infinitesimal form of the multiplicative (Poisson) bivector field $\pi$. The universal lifting theorem of \cite{pont-geng-xu} says that for any multiplicative $k$, $l$ vector fields
$\Pi$ and $\Pi'$ on a Lie groupoid $G \rightrightarrows M$ with Lie algebroid $A$, the corresponding infinitesimal $k, l$ derivations $\delta_\Pi$ and $\delta_{\Pi'}$ on the Gerstenhaber algebra $(\Gamma(\largewedge^\bullet A), \wedge, [~,~])$ satisfies
$$[\delta_\Pi, \delta_{\Pi'}] = \delta_{[\Pi, \Pi']},$$
where the bracket on the left hand side is the graded commutator bracket on the space of multi-derivations on $(\Gamma(\largewedge^\bullet A), \wedge, [~,~])$.

\medskip

\begin{defn} \cite{geng-stie-xu1, geng-stie-xu2, burs-drum}\label{mul-gpd}(multiplicative $(1,1)$-tensors on Lie groupoids)
 Let $G \rightrightarrows M$ be a Lie groupoid. A {\sf multiplicative $(1,1)$-tensor} $N$ on the groupoid is a pair $(N, N_M)$ of
$(1,1)$-tensors on $G$ and $M$, respectively, such that
\[
\xymatrixrowsep{0.4in}
\xymatrixcolsep{0.6in}
\xymatrix{
TG \ar[r]^{N} \ar[d] \ar@<-4pt>[d] & TG \ar[d] \ar@<-4pt>[d]\\
TM  \ar[r]_{N_M} & TM
}
\]
is a Lie groupoid morphism from the tangent Lie groupoid to itself.
\end{defn}

In this case, $N_M$ is the restriction of $N$ to the unit space $M$. Thus, $N_M$ is completely determined by $N$. Hence we may use $N$
to denote a multiplicative $(1,1)$-tensor.

\begin{remark}\label{lie-gp-multiplicative-tensor}
(i) It follows from Definition \ref{mul-gpd} that an $(1,1)$-tensor $N : TG \rightarrow TG$ on the groupoid is multiplicative if and only if $$N (X_g \bullet Y_h) = N(X_g) \bullet N(Y_h),$$ for all 
$X_g \in T_gG, ~ Y_h \in T_hG, ~ g, h \in G$ such that $t_* (X_g) = s_* (Y_h)$, where  $\bullet$ denotes the multiplication of the tangent Lie groupoid $TG \rightrightarrows TM$. 
Thus, it also turns out that the dual $N^* : T^*G \rightarrow T^*G$ is a Lie groupoid morphism from the cotangent Lie groupoid
$T^*G \rightrightarrows A^*$ to itself.

\medskip

 (ii) If $G$ is a Lie group (that is, $M$ is a point), then a $(1,1)$-tensor $N: TG \rightarrow TG$ is multiplicative if $N$ is a Lie group homomorphism from
$TG$ to itself, where $TG$ is equipped with the tangent Lie group structure. That is, $$N (X_g \cdot Y_h) = N(X_g) \cdot N(Y_h),$$ for all
$X_g \in T_gG ,~ Y_h \in T_hG$, $g,h \in G$. Thus, it follows that the multiplicativity of $N$ is being equivalent to  $N \circ m_* = m_* \circ (N \oplus N)$, where $m: G \times G \rightarrow G$ is the Lie group
multiplication map (see also \cite{burs-drum}).
\end{remark}

\begin{defn} \cite{geng-stie-xu1, geng-stie-xu2}(infinitesimal multiplicative $(1,1)$-tensors)
 Let $A$ be a Lie algebroid over $M$. An {\sf infinitesimal multiplicative $(1,1)$-tensor} on $A$ is a pair $(N_A, N_M)$ of $(1,1)$-tensors on $A$
and $M$, respectively, such that
\[
\xymatrixrowsep{0.4in}
\xymatrixcolsep{0.6in}
\xymatrix{
TA \ar[r]^{N_A} \ar[d] & TA \ar[d] \\
TM  \ar[r]_{N_M} & TM
}
\]
is a Lie algebroid morphism from the tangent Lie algebroid $TA \rightarrow TM$ to itself.
\end{defn}

In this case, $N_M$ is the restriction of $N_A$ to the zero section. Thus, $N_M$ is completely determined by $N_A$. Hence we may use $N_A$
to denote an infinitesimal multiplicative $(1,1)$-tensor.

\medskip

Given a multiplicative $(1,1)$-tensor $N$ on a Lie groupoid $G \rightrightarrows M$, define
$$\mathbb{T} N = \sigma^{-1} \circ T N \circ \sigma ,$$
where $\sigma : T(TG) \rightarrow T(TG)$ is the canonical involution map. Then $\mathbb{T} N $ defines  an $(1,1)$-tensor on $TG$.
Moreover, the submanifold $T_M^sG (= A) \hookrightarrow TG$ is an invariant submanifold with respect to the tensor $\mathbb{T}N$.
Therefore, the restriction $ N_A := (\mathbb{T} N)_A$ defines an $(1,1)$-tensor on $A$ (as a manifold).
It also turns out that $N_A$ is an infinitesimal multiplicative $(1,1)$-tensor on $A$. The map
$$ N \rightarrow N_A $$
satisfies the following \cite{geng-stie-xu1, geng-stie-xu2}.
\begin{thm}\label{mul-inf-mul}
 Let $G \rightrightarrows M$ be a $s$-connected and $s$-simply connected Lie groupoid with its Lie algebroid $A$. Then there is a one-to-one
correspondence between multiplicative $(1,1)$-tensors on the groupoid $G \rightrightarrows M$ and infinitesimal multiplicative $(1,1)$-tensors on $A$.
Moreover, this correspondence extends between 
multiplicative Nijenhuis tensors on the groupoid and infinitesimal multiplicative Nijenhuis tensors on the algebroid.
\end{thm}

\section{(1,1)-tensors}\label{sec3}
In this section, we prove some useful results about $(1,1)$-tensors on a manifold or on a Lie groupoid.
\begin{prop}\label{image-nij-submnfld}
 Let $M_1$ and $M_2$ be two smooth manifolds with $(1,1)$-tensors $N_1 : TM_1 \rightarrow TM_1$ and $N_2 : TM_2 \rightarrow TM_2$, respectively. Let $\phi : M_1 \rightarrow M_2$ be a smooth map 
which commute with the $(1,1)$-tensors, that is, $N_2 \circ \phi_* = \phi_* \circ N_1$. If $S_1 \hookrightarrow M_1$ is an invariant submanifold with respect to $N_1$, $\phi(S_1)$ is a submanifold of $N_2$ and $T(\phi(S_1)) = \phi_* (TS_1)$, then 
 $\phi(S_1) \hookrightarrow M_2$ is an invariant submanifold with respect to $N_2$.
\end{prop}

\begin{proof}
We have
\begin{align*}
 N_2 (T(\phi(S_1))) =~ N_2 (\phi_* (TS_1))
=&~ \phi_* (N_1 (TS_1))\\
\subseteq&~ \phi_* (TS_1) = ~ T(\phi(S_1)),
\end{align*}
which shows that $\phi(S_1) \hookrightarrow M_2$ is an invariant submanifold.
\end{proof}

\begin{prop}
 Let $(M_1, N_1)$ and $(M_2, N_2)$ be two smooth manifolds with $(1,1)$-tensors and $S_2 \hookrightarrow M_2$ is an invariant submanifold with respect to $N_2$.
If $\phi : M_1 \rightarrow M_2$ is a smooth map commute with the $(1,1)$-tensors and transverse to $S_2$, then $\phi^{-1}(S_2) \hookrightarrow M_1$ is an invariant
submanifold with respect to $N_1$.
\end{prop}

\begin{proof}
 Since $\phi$ is transverse to $S_2$, the inverse $\phi^{-1}(S_2)$ is a submanifold of $M_1$. Moreover, we have
$T(\phi^{-1}(S_2)) = (\phi_*)^{-1}(TS_2)$. Therefore,
\begin{align*}
 N_1 (T(\phi^{-1}(S_2))) = N_1 ( (\phi_*)^{-1}(TS_2)  ) \subseteq &~ (\phi_*)^{-1} N_2 (TS_2) \\
\subseteq&~ (\phi_*)^{-1} (TS_2) = T(\phi^{-1}(S_2)).
\end{align*}
This shows that $\phi^{-1}(S_2) \hookrightarrow M_1$ is an invariant submanifold with respect to $N_1$.
\end{proof}

\begin{prop}\label{coinduced-nij-tensor}
 Let $M_1$ and $M_2$ be two smooth manifolds and $\phi : M_1 \rightarrow M_2$ be a surjective submersion. Let $N_1 : TM_1 \rightarrow TM_1$ be an $(1,1)$-tensor on $M_1$. Then there
exists a unique $(1,1)$-tensor $N_2 : TM_2 \rightarrow TM_2$ such that $\phi$ commute with the $(1,1)$-tensors, that is, $N_2 \circ \phi_* = \phi_* \circ N_1$ if and only if the kernel pair 
$R (\phi)$ of $\phi$ defined by
 $$R (\phi) = \{ (x, y) |~ \phi(x) = \phi (y)\} \subset M_1 \times M_1$$  is an invariant submanifold with respect to $N_1 \oplus N_1$.
\end{prop}

\begin{proof}
Suppose $(x^1, \ldots, x^n)$ and $(y^1, \ldots, y^n)$ denote the coordinates on the two copies of $M_1$ and $(z^1, \ldots, z^m)$ on $M_2$ ($n=$ dim $M_1 \geq$ dim $M_2 = m$)
such that $\phi$ has the canonical equations $\phi (x^i) = z^i$ and $\phi (y^i) = z^i$, for $i = 1, \ldots, m$.
This shows that $R (\phi)$ is an embedded submanifold of $M_1 \times M_1$ with local coordinates $(x^1, \ldots, x^m, x^{m+1}, \ldots, x^n, y^{m+1}, \ldots, y^n)$.
Moreover, for any $(x,y) \in R(\phi)$, we have
$$T_{(x,y)} R(\phi) = \{  (X_x, Y_y) \in T_x M_1 \times T_y M_1 | ~ \phi_* (X_x) = \phi_* (Y_y) \} .$$
Therefore,  $R(\phi)$ is an invariant submanifold of $M_1 \times M_1$ with respect to the $(1,1)$-tensor $N_1 \oplus N_1$ if and only if
$$(N_1 \oplus N_1)(X_x, Y_y) = (N_1(X_x) , N_1(Y_y))$$
is in $T_{(x,y)}(R(\phi))$, for all $(X_x, Y_y) \in T_{(x,y)}(R(\phi))$. This is equivalent to $\phi_* (N_1(X_x)) = \phi_* (N_1(Y_y))$, for all $(X_x, Y_y) \in T_{(x,y)}(R(\phi))$. Then 
for any $p \in M_2$ and $V_p \in T_pM_2$, the map
$N_2 : TM_2 \rightarrow TM_2$ is defined by
$N_2 (V_p) = \phi_* N_1 (X_x)$, for any $x \in \phi^{-1}(p)$ and $X_x \in (\phi_*)^{-1}(V_p)$.
The map $N_2$ is well defined and satisfies the
required property.
\end{proof}

\begin{prop}\label{nij-map-grpah-nij-submnfld}
 Let $(M_1, N_1)$ and $(M_2, N_2)$ be two smooth manifolds with $(1,1)$-tensors and $\phi : M_1 \rightarrow M_2$ be a smooth map. Then $\phi$ commute with the $(1,1)$-tensors, that is, $N_2 \circ \phi_* = \phi_* \circ N_1$ if and only if
$$\text{Gr} (\phi) = \{ (x, \phi(x))| ~ x \in M_1 \}$$ is an invariant submanifold of $M_1 \times M_2$ with respect to $N_1 \oplus N_2$.
\end{prop}

\begin{proof}
 Note that $\text{Gr} (\phi)$ is a closed embedded submanifold of $M_1 \times M_2$. A tangent vector to the graph at the point $(x, \phi(x))$ consists of a pair $(X_x, \phi_* (X_x))$,
where $x \in M_1$ , $X_x \in T_xM_1$. Thus, $\text{Gr}(\phi)$ is an invariant submanifold of $M_1 \times M_2 $ with respect to the tensor $N_1 \oplus N_2)$ if and only if 
$$( N_1 \oplus N_2 )(X_x , \phi_* (X_x)) = (N_1 (X_x), N_2 (\phi_* (X_x)))$$
is in $T_{(x, \phi(x))}\text{Gr}(\phi)$, for all $X_x \in T_xM_1,~ x \in M_1$. This is equivalent to $N_2 \circ \phi_* (X_x) = \phi_* \circ N_1 (X_x)$, for all $X_x \in T_xM_1,~ x \in M_1$.
Hence the result follows.
\end{proof}

Let $G$ is a Lie group with a multiplicative $(1,1)$-tensor $N: TG \rightarrow TG$. This is equivalent to  $N \circ m_* = m_* \circ (N \oplus N)$, where $m : G \times G \rightarrow G$
is the group multiplication (cf. Remark \ref{lie-gp-multiplicative-tensor}).
Thus, by Proposition \ref{nij-map-grpah-nij-submnfld}, it follows that the multiplicativity of $N$ is being equivalent to the graph 
$$ \text{Gr}(m) = \{ (g, h, gh)| ~ g, h \in G \}$$
of the group multiplication map is an invariant submanifold of $G \times G \times G$ with respect to $N \oplus N \oplus N$.

More generally, we can prove the following result.
\begin{prop}\label{nij-grpd-graph-inv}
 Let $G \rightrightarrows M$ be a Lie groupoid and $N : TG \rightarrow TG$ be an $(1,1)$-tensor on $G$. Then $N$ is multiplicative
if and only if the graph 
$$ \text{Gr}(m) = \{ (g, h, gh)| ~ t(g) = s(h)\} $$
of the groupoid multiplication map is an invariant submanifold of $G \times G \times G$ with respect to the tensor $N \oplus N \oplus N.$
\end{prop}

\begin{proof}
 Let $(g, h , gh) \in \text{Gr}(m)  \subset G \times G \times G$. Then a vector tangent to $\text{Gr}(m)$ at the point $(g, h, gh)$ is of the form $(X_g, Y_h, X_g \bullet Y_h)$,
where $X_g \in T_gG,~  Y_h \in T_hG$ are such that $t_* (X_g)= s_* (Y_h)$. Therefore, $\text{Gr}(m)$ is an invariant submanifold of $G \times G \times G$
with respect to the tensor $N \oplus N \oplus N$ if and only if
$$(N \oplus N \oplus N)( X_g, Y_h, X_g \bullet Y_h ) = ( N(X_g) , N(Y_h), N(X_g \bullet Y_h) )$$
is in $T_{(g,h,gh)} \text{Gr}(m)$. This is equivalent to the condition that $N (X_g \bullet Y_h) = N(X_g) \bullet N(Y_h),$
that is, $N$ is a Lie groupoid morphism from the tangent Lie groupoid to itself.
\end{proof}

\begin{defn}
 A {\sf Nijenhuis groupoid} is a Lie groupoid $G \rightrightarrows M$ together with a multiplicative Nijenhuis tensor $N$ (cf. Definition \ref{mul-gpd}). A Nijenhuis groupoid may also be denoted
by $(G \rightrightarrows M, N)$.
\end{defn}

\begin{exam}
(i) Any Lie groupoid with the identity Nijenhuis tensor is a Nijenhuis groupoid.

\medskip

(ii) Let $M$ be a smooth manifold with a Nijenhuis tensor $N : TM \rightarrow TM$. Then the pair groupoid $M \times M \rightrightarrows M$ is a Nijenhuis 
groupoid with Nijenhuis tensor $N \oplus N$.
\end{exam}

Let $(G \rightrightarrows M, N)$ be a Nijenhuis groupoid. Since $N : TG \rightarrow TG$ is a Lie groupoid morphism from the tangent Lie groupoid to itself, we have
$$ s_* \circ N = N_M \circ s_*  ~~ \text{ and }~~ t_* \circ N = N_M \circ t_*.$$
Therefore, it follows that the graph Gr$((s,t)) = \{ (g, s(g), t(g)) |~ g \in G \} \subset G \times M \times M$ of the map
$(s,t) : G \rightarrow M \times M , ~ g \mapsto (s(g), t(g))$ is an invariant submanifold with respect to the tensor $N \oplus N_M \oplus N_M$.
Hence by Proposition \ref{nij-map-grpah-nij-submnfld}, the map $(s,t) : G \rightarrow M \times M$ commute with the Nijenhuis tensors, that is,
$$ (s,t)_* \circ N = (N_M \oplus N_M) \circ (s,t)_*.$$


\begin{defn}
 Let $(G \rightrightarrows M, N)$ be a Nijenhuis groupoid. A Lie subgroupoid $H \rightrightarrows S$ is called an {\sf invariant subgroupoid}
of $(G \rightrightarrows M, N)$ if $H \hookrightarrow G$ is an invariant submanifold with respect to $N$.
\end{defn}

\begin{defn}
 Let $A \rightarrow M$ be a Lie algebroid and $N_A : TA \rightarrow TA$ be an infinitesimal multiplicative Nijenhuis tensor on $A$. A Lie subalgebroid
$B \rightarrow S$ of $A \rightarrow M$ is called an {\sf invariant subalgebroid} if $B \hookrightarrow A$ is an invariant submanifold with respect to $N_A$.
\end{defn}

\begin{prop}\label{subalgbd-base-inv}
 Let $A \rightarrow M$ be a Lie algebroid and $N_A : TA \rightarrow TA$ be an infinitesimal multiplicative Nijenhuis tensor on $A$ with base map $N_M : TM \rightarrow TM$. If $B \rightarrow S$
is an invariant subalgebroid as above, then $S \hookrightarrow M$ is an invariant submanifold with respect to $N_M$.
\end{prop}

\begin{proof}
 Note that, $N_M = N_A |_0$ is the restriction of $N_A$ to the zero section and $TS = (TB)|_0$ is the zero section of the bundle
$TB \rightarrow TS$. Therefore, we have
$$ N_M (TS) =  N_A |_0 ( (TB)|_0 ) = (N_A (TB))|_0 \subseteq (TB)|_0 = TS.$$
Hence the result follows.
\end{proof}

\begin{prop}\label{subgrpd-subalgbd}
 Let $(G \rightrightarrows M, N)$ be a Nijenhuis groupoid with Lie algebroid $A \rightarrow M$. If $H \rightrightarrows S$ is an invariant subgroupoid of $(G \rightrightarrows M, N)$ with Lie algebroid
$B \rightarrow S$, then $B \rightarrow S$ is an invariant subalgebroid of $A \rightarrow M$.
\end{prop}

\begin{proof}
 Since $H \rightrightarrows S$ is a Lie subgroupoid of $G \rightrightarrows M$, it follows that $B \rightarrow S$ is a Lie subalgebroid of $A \rightarrow M$.
Moreover, we have
$$ B = A |_S \cap TH|_S.$$
Since $\mathbb{T} N = \sigma^{-1} \circ TN \circ \sigma$ is tangent to the submanifold $A$, the restriction
$N_A = (\mathbb{T} N)_A$ is tangent to $A|_S$. Moreover, since 
$H \hookrightarrow G$ is an invariant submanifold with respect to $N$, $TH \hookrightarrow TG$ is an invariant submanifold with respect to $\mathbb{T} N$.
Therefore, it follows that $B \hookrightarrow A$ is an invariant submanifold with respect to $N_A$.
\end{proof}

Thus, by combinding Propositions \ref{subgrpd-subalgbd} and \ref{subalgbd-base-inv}, we get the following.
\begin{corollary}
 Let $(G \rightrightarrows M, N)$ be a Nijenhuis groupoid and $H \rightrightarrows S$ an invariant subgroupoid. Then $S \hookrightarrow M$ is an invariant submanifold
of $M$ with respect to $N_M$.
\end{corollary}

\section{Compatible Lie algebroids}\label{sec4}
In this section, we study compatible Lie algebroid structures on a vector bundle.
\begin{defn}\label{defn-comp-lie-algbds}
 Given a vector bundle $A$ over $M$, two Lie algebroid structures $([~,~]_1, \rho_1)$ and $([~,~]_2, \rho_2)$ on $A$ are said to be {\sf compatible} if
$$ [~,~] = [~, ~]_1 + [~, ~]_2 \text{ , } \rho = \rho_1 + \rho_2$$
also defines  a Lie algebroid structure on $A$.
\end{defn}
In this case, for each $\lambda \in \mathbb{R}$, the bracket
$ [~,~]^\lambda = [~, ~]_1 + \lambda [~, ~]_2$ and the anchor $\rho^\lambda = \rho_1 + \lambda \rho_2$
defines a Lie algebroid structure on $A$.

It is easy to check that the compatibility condition is equivalent to the followings
\begin{equation}\label{comp-1}
 J(X, Y, Z) :=  \big( [[X,Y]_1, Z]_2 +   [[X,Y]_2, Z]_1 \big) + c.p. = 0 , 
\end{equation}
\begin{equation}\label{comp-2}
 \Theta (X, Y) := [\rho_1 (X), \rho_2 (Y)] + [\rho_2 (X), \rho_1 (Y)] - \rho_1 [X, Y]_2 - \rho_2 [X,Y]_1 = 0,
\end{equation}
for all $X, Y, Z \in \Gamma A$, where $c.p.$ means the cyclic permutation of $(X, Y, Z).$

\medskip

It is known that Lie algebroid structures on a vector bundle $A$ over $M$ are in a one-to-one correspondence with homological vector fields of degree
one on the supermanifold $(M, \largewedge^\bullet A^*)$ \cite{vain}. Let $([~,~], \rho)$ be a Lie algebroid structure on $A$ and $(x)$ be a chart of $M$ over which the bundle $A$ is trivial. Given a local basis $(\varepsilon_i)$
of sections of $A$ with dual basis $(\xi^i)$ of $A^*$, suppose
$$ [\varepsilon_i, \varepsilon_j] = \sum c^k_{ij} \varepsilon_k ~~\text{ and }~~ \rho(\varepsilon_i) = \sum a^\alpha_i \partial_{x^\alpha} ,$$
where $c^k_{ij}$'s and $a^\alpha_i$'s are structure functions for the Lie algebroid $(A, [~,~], \rho)$ with respect to the local basis $(\varepsilon_i)$. Then the homological vector field $V$ on $(M, \largewedge^\bullet A^*)$ corresponding to the Lie algebroid structure
is given by
$$ V = \sum c^k_{ij} \xi^i \xi^j \partial_{\xi^k} + \sum a_i^\alpha \xi^i \partial_{x^\alpha} .$$

We have the following equivalent characterizations of compatible Lie algebroid structures.

\begin{thm}\label{comp-equiv}
 Let $A$ be a vector bundle over $M$ and $([~,~]_1, \rho_1) , ([~,~]_2, \rho_2)$ be two Lie algebroid structures on $A$.
If $d_1, d_2$ are the corresponding Lie algebroid
differentials on $\Gamma{(\largewedge^\bullet A^*)}$ and $\pi_1, \pi_2$ are the corresponding dual linear Poisson structures on $A^*$, then the followings
are equivalent
\begin{itemize}
 \item[(i)] $([~,~]_1, \rho_1)$ and $([~,~]_2, \rho_2)$ are compatible$;$
 \item[(ii)] differentials $d_1 ,d_2 : \Gamma{(\largewedge^\bullet A^*)} \rightarrow \Gamma{(\largewedge^{\bullet +1}A^*)}$ are compatible
in the sense that their graded commutator $[d_1, d_2] = 0$, that is,
$$ d_1 d_2 + d_2 d_1 = 0 ;$$
 \item[(iii)] $\pi_1$ and $\pi_2$ are compatible Poisson structures on $A^*$, that is,
$ [\pi_1, \pi_2] = 0 ;$
 \item[(iv)] corresponding homological vector fields $V_1$ and $V_2$ on the supermanifold $(M, \largewedge^\bullet A^*)$ are compatible, that is, $[V_1, V_2] = 0.$
\end{itemize}
\end{thm}

\begin{proof}
 (i) $\Longrightarrow$ (ii): For any $f \in C^\infty(M)$ and $X, Y \in \Gamma A$, we have
\begin{align*}
 &(d_1 d_2 + d_2 d_1)(f) (X, Y) \\
=&~  \rho_1 (X)  \langle  d_2f , Y  \rangle - \rho_1 (Y)  \langle  d_2f , X \rangle - \langle  d_2f , [X,Y]_1 \rangle 
 + \rho_2 (X)  \langle  d_1f , Y  \rangle - \rho_2 (Y)  \langle  d_1f , X \rangle - \langle  d_1f , [X,Y]_2 \rangle \\
=&~ \rho_1(X) \rho_2 (Y) f - \rho_1(Y) \rho_2(X) f - \rho_2 ([X,Y]_1)f 
 + \rho_2(X) \rho_1 (Y) f - \rho_2(Y) \rho_1(X) f - \rho_1 ([X,Y]_2)f \\
=&~  \Theta (X, Y)(f) = 0 ~~~~~ \quad (\text{by using Eqn. } (\ref{comp-2})).
\end{align*}
For any $\alpha \in \Gamma{A^*}$ and $X, Y, Z \in \Gamma A$, we also have
\begin{align*}
 &(d_1 d_2 + d_2 d_1)(\alpha) (X, Y, Z) \\
=&~ \rho_1 (X) \cdot \langle d_2 (\alpha) (Y,Z) \rangle - \rho_1 (Y) \cdot \langle d_2 (\alpha) (X, Z) \rangle + \rho_1 (Z) \cdot \langle d_2 (\alpha) (X, Y)  \rangle\\
& -d_2(\alpha) ([X,Y]_1, Z) + d_2(\alpha) ([X,Z]_1, Y) - d_2 (\alpha)([Y, Z]_1, X) \\
&+ \rho_2 (X) \cdot \langle d_1 (\alpha) (Y,Z) \rangle - \rho_2 (Y) \cdot \langle d_1 (\alpha) (X, Z) \rangle + \rho_2 (Z) \cdot \langle d_1 (\alpha) (X, Y)  \rangle\\
& -d_1(\alpha) ([X,Y]_2, Z) + d_1(\alpha) ([X,Z]_2, Y) - d_1 (\alpha)([Y, Z]_2, X) \\
=&~ \Theta (X, Y) \langle \alpha, Z \rangle + \Theta (Y, Z) \langle  \alpha, X \rangle + \Theta (Z, X) \langle \alpha, Y \rangle + \alpha (J(X, Y, Z)) \\
=&~ 0 ~~ \quad (\text{by using Equations } (\ref{comp-1}) \text{ and } (\ref{comp-2})).
\end{align*}
Thus, it follows that, $d_1 d_2 + d_2 d_1 = 0$ on $C^\infty(M)$ and $\Gamma A^*$.
Moreover, since
$$ (d_1 d_2 + d_2 d_1 ) (\alpha \wedge \beta) = (d_1 d_2 + d_2 d_1) \alpha \wedge \beta + \alpha \wedge (d_1 d_2 + d_2 d_1) \beta ,$$
for any $\alpha, \beta \in \Gamma {(\largewedge^\bullet A^*)}$, we have $d_1 d_2 + d_2 d_1 \equiv 0$ on $\Gamma {(\largewedge^\bullet A^*)}$.

\medskip
(ii) $\Longrightarrow$ (i): This follows from the previous calculation.


\medskip
(i) $\Longleftrightarrow$ (iii): Let $q_*$ denote the projection map of the bundle $A^*$ over $M$. Then for any $f \in C^\infty(M)$ and $X, Y, Z \in \Gamma A$, we have
\begin{align*}
 &[\pi_1 , \pi_2 ] (l_X, l_Y, l_Z)\\
=&~ \pi_1 (\pi_2 (l_X, l_Y) , l_Z) - \pi_1 ( \pi_2 (l_X , l_Z), l_Y) + \pi_1 ( \pi_2 (l_Y, l_Z), l_X) \\
& + \pi_2 (\pi_1 (l_X, l_Y) , l_Z) - \pi_2 ( \pi_1 (l_X , l_Z), l_Y) + \pi_2 ( \pi_1 (l_Y, l_Z), l_X) \\
=&~ l_{[[X,Y]_2, Z]_1} - l_{[[X,Z]_2, Y]_1} + l_{[[Y,Z]_2 , X]_1} + l_{[[X,Y]_1, Z]_2} - l_{[[X,Z]_1, Y]_2} + l_{[[Y,Z]_1 , X]_2} \\
=&~ l_{J(X, Y, Z)},
\end{align*}
and
\begin{align*}
 &[\pi_1 , \pi_2] (l_X, l_Y, f \circ q_*) \\
=&~ \pi_1 (\pi_2 (l_X, l_Y) , f \circ q_*) - \pi_1 ( \pi_2 (l_X , f \circ q_*), l_Y) + \pi_1 ( \pi_2 (l_Y, f \circ q_*), l_X) \\
& + \pi_2 (\pi_1 (l_X, l_Y) , f \circ q_*) - \pi_2 ( \pi_1 (l_X , f \circ q_*), l_Y) + \pi_2 ( \pi_1 (l_Y, f \circ q_*), l_X) \\
=&~ (\rho_1 ([X,Y]_2) f) \circ q_* + (\rho_1  (Y) \rho_2 (X) f) \circ q_*  - (\rho_1(X) \rho_2(Y)f) \circ q_* \\
& + (\rho_2 ([X,Y]_1) f) \circ q_* + (\rho_2  (Y) \rho_1 (X) f) \circ q_*  - (\rho_2(X) \rho_1(Y)f) \circ q_* \\
=&~ - \Theta(X,Y)(f) \circ q_* .
\end{align*}
Since the space of smooth functions on $A^*$ are generated by linear and pull back functions, it follows from the above observation that
(i) is equivalent to (iii).

\medskip
(i)  $\Longrightarrow$ (iv): The homological vector field corresponding to the Lie algebroid structure $([~,~] = [~,~]_1 + [~,~]_2, \rho = \rho_1 + \rho_2)$
is given by
$$ \sum (c^k_{ij} + d^k_{ij}) \xi^i \xi^j \partial_{\xi^k} + \sum (a_i^\alpha + b_i^\alpha) \xi^i \partial_{x^\alpha}  = V_1 + V_2 ,$$
where $\{c^k_{ij}, a^\alpha_i \}$ and $\{d^k_{ij}, b^\alpha_i \}$ are structure functions for Lie algebroid structures $([~,~], \rho_1)$ and $([~,~]_2, \rho_2)$, respectively.
Therefore,
$$ 0 = [V_1 + V_2 , V_1 + V_2] = [V_1,V_1] + [V_1,V_2] + [V_2,V_1] + [V_2,V_2] = 2 [V_1 , V_2] .$$
The converse part (iv) $\Longrightarrow$ (i) is similar.
\end{proof}

\begin{remark}
 If $([~,~]_1, \rho_1)$ and $([~,~]_2, \rho_2)$ are compatible Lie algebroid structures on $A$, then the differential of the Lie algebroid $(A, [~,~] = [~,~]_1 + [~,~]_2, \rho = \rho_1 + \rho_2)$ is given by
$d = d_1 + d_2$, and the dual linear Poisson structure on $A^*$ is given by $\pi = \pi_1 + \pi_2$.
\end{remark}

\begin{exam}\label{comp-lie-exam}
(i) Compatible Lie algebra structures on a vector space are examples of compatible Lie algebroids over a point \cite{bai-wu}.

\medskip

(ii) (Nijenhuis tensors)

Let $M$ be a smooth manifold and $N : TM \rightarrow TM$ be a Nijenhuis tensor on $M$. 
Then for each $k \geq 0$, the hierarchy of Lie algebroid structures $([~,~]_{N^k}, id_{N^k}= N^k)$ on the tangent bundle $TM$ are  pairwise compatible \cite{yks-mag}.
Thus, in particaular,
the tangent bundle $TM$ with its usual Lie algebroid structure
$([~,~], id)$ and deformed Lie algebroid structure $([~,~]_N , id_N)$ are compatible.

\medskip

(iii) (compatible Poisson structures)

Two Poisson structures $\pi$ and $\pi'$ on a manifold $M$ are compatible if and only if the cotangent Lie algebroid structures $(T^*M)_\pi$ and $(T^*M)_{\pi'}$ are
compatible. If $d_{\pi}$ and $d_{\pi'}$ are the corresponding Lie algebroid differentials, then 
$$d_{\pi} = [\pi, \_] ~\text{~~ and ~~}~ d_{\pi'} = [\pi', \_] .$$ 
Thus, the result follows from the following observation that 
\begin{align*}
 [d_{\pi} , d_{\pi'}] = d_{\pi} \circ d_{\pi'} + d_{\pi'} \circ d_{\pi} 
= [ \pi , [\pi' , \_]] + [ \pi' , [\pi , \_]] 
= [[\pi , \pi'], \_],
\end{align*}
where the last equality follows from the graded Jacobi identity of the Schouten bracket of multivector fields.

As for example, if $(\pi, N)$ is a Poisson-Nijenhuis structure on $M$ with hierarchy of compatible Poisson structures $\pi_k = N^k \pi$, for $k \geq 0$, then $(T^*M)_{\pi_k}$
and $(T^*M)_{\pi_l}$ are compatible Lie algebroid structures on $T^*M$.

\medskip

(iv) (compatible Jacobi structures)

A {\sf Jacobi structure} on a manifold $M$ consists of a pair $(\pi, E)$ of a bivector field $\pi$ and a vector field $E$ on $M$ satisfying 
\begin{equation}\label{jac-manifold}
[\pi, \pi] = 2 E \wedge \pi, \quad [E, \pi] = 0.
\end{equation}
Note that (\ref{jac-manifold}) is equivalent to the following condition
\begin{equation}\label{jac-alt-defn}
 [(\pi, E), (\pi, E)]^{(0,1)} = 0,
\end{equation}
where $(\pi, E) \in \Gamma (\largewedge^2TM) \times \Gamma (TM) = \Gamma (\largewedge^2(TM \times \mathbb{R}))$ is considered as a $2$-multisection of the Lie algebroid
$TM \times \mathbb{R} \rightarrow M$ and $[~, ~]^{(0,1)}$ is the Gerstenhaber bracket on the multisections of the Lie algebroid $TM \times \mathbb{R} \rightarrow M$
deformed by the $1$-cocycle $(0,1) \in \Gamma(T^*M) \times C^\infty(M) = \Gamma(T^*M \times \mathbb{R})$ \cite{igl-mar}. Given a Jacobi structure $(\pi, E)$ on $M$,
the $1$-jet bundle $T^*M \times \mathbb{R} \rightarrow M$ carries a Lie algebroid structure whose bracket and anchor are given by
$$ [(\alpha, f),(\beta, g)]_{(\pi,E)} := \mathcal{L}_{(\pi,E)^\sharp(\alpha, f)}^{(0,1)} (\beta, g) - \mathcal{L}_{(\pi,E)^\sharp(\beta, g)}^{(0,1)} (\alpha, f) - d^{(0,1)}\big((\pi,E)((\alpha, f),(\beta, g))\big),$$
$$\rho_{(\pi,E)} (\alpha, f) := \pi^\sharp(\alpha) + f E, $$
for all $(\alpha, f), (\beta, g) \in \Gamma(T^*M) \times C^\infty(M) = \Gamma(T^*M \times \mathbb{R})$ \cite{igl-mar}. 
Here $d^{(0,1)}$ and $\mathcal{L}^{(0,1)}$ denotes the $(0,1)$-twisted differential and Lie derivative of the Lie algebroid $TM \times \mathbb{R}$.
The differential
$ d_{(\pi,E)}$ 
of the Lie algebroid $(T^*M \times \mathbb{R}, [~,~]_{(\pi,E)}, \rho_{(\pi,E)})$ is given by $$d_{(\pi,E)} (P,Q) = [(\pi,E), (P,Q)]^{(0,1)}, $$
for $(P,Q) \in \Gamma(\largewedge^\bullet TM ) \times \Gamma(\largewedge^{\bullet -1} TM ) = \Gamma(\largewedge^\bullet(TM \times \mathbb{R}))$. 

Two Jacobi structures
$(\pi_1, E_1)$ and $(\pi_2,E_2)$ are said to be {\sf compatible} if $(\pi_1 + \pi_2, E_1 +E_2)$ is a Jacobi structure on $M$ \cite{costa}. Thus, by using (\ref{jac-alt-defn}), compatibility
is being equivalent to the condition
$$[(\pi_1,E_1), (\pi_2, E_2)]^{(0,1)} = 0 .$$
Two Jacobi structures on $M$ are compatible if and only if their corresponding Lie algebroid structures on the $1$-jet bundle $T^*M \times \mathbb{R} \rightarrow M$ are compatible. This follows from the observation that
\begin{align*}
&d_{(\pi_1,E_1)} \circ d_{(\pi_2,E_2)} + d_{(\pi_2,E_2)} \circ d_{(\pi_1,E_1)}\\
=~& [(\pi_1,E_1), [(\pi_2, E_2),  \_]^{(0,1)}]^{(0,1)} + [(\pi_2,E_2), [(\pi_1, E_1),  \_]^{(0,1)}]^{(0,1)} \\
=~& [[(\pi_1,E_1), (\pi_2, E_2)]^{(0,1)}, \_]^{(0,1)} = 0.
\end{align*}
\end{exam}

\medskip

\begin{defn}
 Let $A$ be a Lie algebroid over $M$. Then two Lie bialgebroid structures on $A$ are said to be compatible if the corresponding Lie algebroid structures on $A^*$
are compatible.
\end{defn}

\begin{remark}\label{comp-lie-bialgbd-rem}
(i) In terms of degree one differential, two Lie bialgebroid structures $(A, \delta)$ and $(A, \delta')$ on a Lie algebroid $A$ are compatible if and only if $[\delta, \delta']= 0$.
This follows from Remark \ref{lie-bialgebroid-diff} and Theorem \ref{comp-equiv}.

(ii) If $(A, \delta)$ and $(A, \delta')$ are two compatible Lie bialgebroid structures on a Lie algebroid $A$, then the differential $\delta + \delta'$ also defines a
Lie bialgebroid structure on $A$.
\end{remark}

\begin{exam}
 If $\pi$ and $\pi'$ are two compatible Poisson structures on $M$, then $(TM, (T^*M)_\pi)$ and $(TM, (T^*M)_{\pi'})$ are two compatible Lie bialgebroid structures
on $TM$. 

As for example, if $(\pi, N)$ is a Poisson-Nijenhuis structure on $M$ with hierarchy of compatible Poisson structures 
 $\pi_k = N^k \pi$, for $k \geq 0$, then $(TM , (T^*M)_{\pi_k})$
and $(TM, (T^*M)_{\pi_l})$ are compatible Lie bialgebroid structures on $TM$.
\end{exam}

 


\section{Poisson-Nijenhuis groupoids}\label{sec5}
In this section, we define multiplicative Poisson-Nijenhuis structures on a Lie groupoid which extends the notion of symplectic-Nijenhuis groupoids introduced by Sti$\acute{\text{e}}$non and Xu \cite{stienon-xu}.
\begin{defn}\label{pn-grpd}
 A {\sf Poisson-Nijenhuis groupoid} is a Lie groupoid $G \rightrightarrows M$ together with 
a Poisson-Nijenhuis structure $(\pi, N)$ on $G$ such that $(G \rightrightarrows M, \pi)$ forms a Poisson groupoid and $(G \rightrightarrows M, N)$ a Nijenhuis groupoid.
\end{defn}
Thus, a Poisson-Nijenhuis groupoid is a Poisson groupoid $(G \rightrightarrows M, \pi)$ together with a multiplicative Nijenhuis tensor $N : TG \rightarrow TG$
such that $(G, \pi, N)$ is a Poisson-Nijenhuis manifold.
A Poisson-Nijenhuis groupoid as above is denoted by the triple $(G \rightrightarrows M, \pi, N).$

\begin{exam}\label{pn-grpd-exam}
(i) Any Poisson groupoid is a Poisson-Nijenhuis groupoid with $N = id$.

\medskip

(ii) (symplectic-Nijenhuis groupoids)

A {\sf symplectic-Nijenhuis groupoid} is a symplectic groupoid $(G \rightrightarrows M , \omega)$
equipped with a multiplicative Nijenhuis tensor $N : TG \rightarrow TG$ such that $(G, \omega , N)$
is a symplectic-Nijenhuis manifold \cite{stienon-xu}. Thus, a symplectic-Nijenhuis groupoid is a
Poisson-Nijenhuis groupoid whose Poisson structure is non-degenerate.
 
\medskip

(iii) (pair groupoid of Poisson-Nijenhuis manifolds)

Let $(M, \pi, N)$ be a Poisson-Nijenhuis manifold. Consider the pair groupoid $M \times M \rightrightarrows M$ with the Poisson-Nijenhuis structure
$(\pi \ominus \pi, N \oplus N)$ on $M \times M$. Using the structure of pair groupoid, it is straightforward to check that both the Poisson and the Nijenhuis
tensors are multiplicative. Thus, the pair groupoid $(M \times M \rightrightarrows M, \pi \ominus \pi, N \oplus N)$ is a Poisson-Nijenhuis groupoid.

\medskip

(iv) (complemented Poisson groupoids)

 Let $(G \rightrightarrows M, \pi)$ be a Poisson groupoid and $\omega \in \Omega^2(G)$ be a multiplicative
$2$-form on $G$ (by a multiplicative $2$-form on a Lie groupoid $G \rightrightarrows M$, we mean a $2$-form $\omega \in \Omega^2(G)$ such that the induced map
$\omega^\sharp : TG \rightarrow T^*G$ is a Lie groupoid morphism from the tangent Lie groupoid $TG \rightrightarrows TM$ to the cotangent Lie groupoid $T^*G \rightrightarrows A^*$). If $\omega$ satisfies
$$ \iota_{\pi^\sharp \alpha} d\omega = 0, ~~ \text{for all }  \alpha \in \Omega^1(G)$$
(in particaular, if $d\omega = 0$), then the triple $(G \rightrightarrows M, \pi,  N = \pi^\sharp \circ \omega^\sharp)$ forms a Poisson-Nijenhuis groupoid (see Example \ref{pn-manifold-exam} (ii)).

\medskip

(v) (holomorphic Poisson groupoids) 

A {\sf holomorphic Lie groupoid} is a (smooth) Lie groupoid $G \rightrightarrows M$, where both $G$ and $M$ are 
complex manifolds and all the structure maps of the groupoid are holomorphic. If $G \rightrightarrows M$ is a holomorphic Lie groupoid with associated almost complex
structures $J : TG \rightarrow TG$ and $J_M : TM \rightarrow TM$, then $(J, J_M)$ is a multiplicative Nijenhuis tensor on the underlying real smooth Lie groupoid
$G \rightrightarrows M.$

A {\sf holomorphic Poisson groupoid} is a holomorphic Lie groupoid $G \rightrightarrows M$ equipped with a holomorphic Poisson structure
 $\pi = \pi_R + i \pi_I \in \Gamma(\largewedge^2 T^{1,0}G)$ such that the graph of the groupoid multiplication $\text{Gr}(m) \subset G \times G \times \overline{G}$ is 
a coisotropic submanifold, where $\overline{G}$ is the manifold $G$ with the opposite Poisson structure. In this case, it turns out that $(G \rightrightarrows M, \pi_R)$
and $(G \rightrightarrows M, \pi_I)$ are both Poisson groupoids \cite{jotz-stienon-xu}. Therefore, it follows from Example \ref{pn-manifold-exam} (iii) that $(G \rightrightarrows M, \pi_I, J)$ forms a Poisson-Nijenhuis groupoid.
\end{exam}

\begin{remark}
If $(G \rightrightarrows M, \pi, N)$ is a Poisson-Nijenhuis groupoid, then for each $k \geq 0$, the Poisson bivectors $ \pi_k = N^k \pi$ are multiplicative. Therefore,
 $(G \rightrightarrows M, \pi_k = N^k \pi)$ forms a Poisson groupoid, for each $k \geq 0$. These Poisson groupoid structures on
$ G \rightrightarrows M$ are compatible in the sense that the Poisson structures $\{\pi_k = N^k \pi\}_{k \geq 0}$ on $G$ are compatible.
\end{remark}

\begin{remark}\label{pn-grpd-comp-liebialgbd}
 Let $(G \rightrightarrows M, \pi, N)$ be a Poisson-Nijenhuis groupoid with Lie algebroid $A$. Since $\{ \pi_k = N^k \pi \}_{k \geq 0}$ are compatible multiplicative Poisson
structures on $G$, it follows from the universal lifting theorem \cite{pont-geng-xu} that
$$ [\delta_{\pi_k}, \delta_{\pi_l}] = \delta_{[\pi_k, \pi_l]} = 0, ~~ \text{ for } k , l \geq 0.$$
Therefore, these compatible Poisson structures on $G \rightrightarrows M$ induces compatible Lie bialgebroid structures on $A$.
\end{remark}

\begin{prop}\label{pn-grpd-base-pn}
 Let $(G \rightrightarrows M, \pi, N)$ be a Poisson-Nijenhuis groupoid. Then
\begin{itemize}
 \item[(i)] the groupoid inversion map $i : G \rightarrow G,~ g \mapsto g^{-1}$ is an anti-P-N map;
 \item[(ii)] the base manifold $M$ carries a unique
Poisson-Nijenhuis structure such that $s$ is a P-N map and $t$ is an anti-P-N map.
\end{itemize}
\end{prop}

\begin{proof}
(i) Since $(G \rightrightarrows M, \pi)$ is a Poisson groupoid, the groupoid inversion $i : G \rightarrow G$ is an anti-Poisson map (Theorem (4.2.3), \cite{wein}).
Moreover, since $N : TG \rightarrow TG$ is a groupoid morphism from the tangent Lie groupoid to itself, we have
$$ N(X_g) \bullet N(X_g^{-1}) = N (X_g \bullet X_g^{-1}) = N (\epsilon_* ({s_* (X_g)})) = \epsilon_* ({N_M (s_* (X_g))}) = \epsilon_* ({s_* (N(X_g))}) ~ \text{ and}$$
$$ N(X_g^{-1}) \bullet N(X_g) = N (X_g^{-1} \bullet X_g) =  N ( \epsilon_* ({t_* (X_g)})) = \epsilon_* ({N_M (t_* (X_g))}) = \epsilon_* ({t_* (N(X_g))}) ,$$
for all $X_g \in T_gG,$ $g \in G.$ Here $\epsilon : M \rightarrow G$ denotes the unit map of the Lie groupoid $G \rightrightarrows M$. Thus, it follows that
$N (X_g^{-1}) = (N (X_g))^{-1}$, which is equivalent to
$ N \circ i_* = i_* \circ N.$
Therefore, the inversion map $i$ is an anti-P-N map.
\medskip

(ii) Since $(G \rightrightarrows M, \pi)$ is a Poisson groupoid, the base $M$ carries a Poisson structure
$\pi_M$ such that $s$ is a Poisson map and $t$ is an anti-Poisson map (Theorem (4.2.3), \cite{wein}). Moreover, since $N : TG \rightarrow TG$
is a multiplicative Nijenhuis tensor, $N_M : TM \rightarrow TM$ is a Nijenhuis tensor on $M$. Thus, it remains
show that, $\pi_M$ and $N_M$ are compatible in the sense of Poisson-Nijenhuis structure.

Note that $\pi_M = s_* \pi$ implies that $\pi_M^\sharp = s_* \circ \pi^\sharp \circ s^*$. Moreover, since $N$ is a multiplicative
Nijenhuis tensor, we have $s_* \circ N = N_M \circ s_*$. Therefore,
\begin{align*}
 N_M \circ \pi_M^\sharp = N_M \circ s_* \circ \pi^\sharp \circ s^* =& ~ s_* \circ N \circ \pi^\sharp \circ s^* \\
=& ~s_* \circ \pi^\sharp \circ N^* \circ s^* \qquad (\text{since } \pi \text{ and } N  \text{ are compatible})\\
=& ~s_* \circ \pi^\sharp \circ (s_* \circ N)^* \\
=& ~s_* \circ \pi^\sharp \circ s^* \circ N_M^* = \pi_M^\sharp \circ N_M^* .
\end{align*}
Hence, $N_M \circ \pi_M^{\sharp}$ defines a bivector field $N_M \pi_M$ on $M$ and we have
$(N_M \pi_M)^\sharp = s_* \circ (N \pi)^\sharp \circ s^*$. This shows that $s_* (N \pi) = N_M \pi_M$.
Thus, it follows from an easy observation that for any $\alpha , \beta \in \Omega^1(M)$,
$$ [s^* \alpha, s^* \beta]_{\pi} = s^* ( [\alpha, \beta]_{\pi_M} )  ~~ \text{ and }~~ [s^* \alpha, s^* \beta]_{N \pi} = s^* ( [\alpha, \beta]_{N_M \pi_M}) .$$
Therefore,
\begin{align*}
0 &= C (\pi, N) (s^* \alpha, s^* \beta) \\
&= [s^* \alpha, s^* \beta]_{N \pi} - \big( [N^* s^* \alpha, s^*\beta]_\pi + [s^* \alpha, N^* s^* \beta]_\pi - N^* [s^* \alpha, s^* \beta]_\pi   \big) \\
&= s^* ([\alpha , \beta]_{N_M \pi_M}) - \big( [ s^* N_M^* \alpha, s^*\beta]_\pi + [s^* \alpha, s^* N_M^* \beta]_\pi - N^* s^* [\alpha, \beta]_{\pi_M}   \big)\\
&= s^*   ([\alpha, \beta]_{N_M \pi_M} )
 - s^* \big(  [N_M^* \alpha, \beta]_{\pi_M}    + [\alpha, N_M^*\beta]_{\pi_M}    - N_M^* [\alpha, \beta]_{\pi_M}  \big) \\
&= s^* \big( C(\pi_M , N_M) (\alpha, \beta) \big) .
\end{align*}
Since $s$ is a surjective submersion, it follows that $C(\pi_M , N_M) (\alpha, \beta) = 0.$
Thus, $(\pi_M, N_M)$ defines a Poisson-Nijenhuis structure on $M$.
\end{proof}

\begin{remark}
For a Poisson-Nijenhuis groupoid $(G \rightrightarrows M, \pi, N)$, the Poisson structure on $M$ induced from the Poisson groupoid $(G \rightrightarrows M, N^k \pi)$ is given by
$N_M^k \pi_M.$
\end{remark}

\begin{prop}
 Let $(G \rightrightarrows M, \pi, N)$ be a Poisson-Nijenhuis groupoid. If the orbit space $M/\sim$ is
a smooth manifold, then $ M/\sim $ carries a Poisson-Nijenhuis structure such that the projection
$q : M \rightarrow M/\sim$ is a P-N map.
\end{prop}

\begin{proof}
 Since $(G \rightrightarrows M, \pi)$ is a Poisson groupoid, the orbit space $M/\sim$ carries a Poisson structure $\underline{\pi}$ such that $q : M \rightarrow M/\sim$
is a Poisson map (Corollary (4.2.9), \cite{wein}). For the projection map $q$, we have
\begin{align*}
R (q) =&~ \{ (x, y) \in M \times M | ~ q(x) = q(y) \} \\
=& ~\{ (s(g), t (g))| ~ g \in G \} \\
=&~ (s,t)(G).
\end{align*}
Consider the map $(s,t) : G \rightarrow M \times M$. Note that, $G$ is an invariant submanifold of $G$ with respect to the tensor $N : TG \rightarrow TG$, and we have
$ (s,t)_* \circ N = (N_M \oplus N_M) \circ (s,t)_*.$
Therefore,
by Proposition \ref{image-nij-submnfld}, it follows that $(s,t)(G) \subset M \times M$ is an invariant submanifold with respect to the tensor
$N_M \oplus N_M$. Hence by
Proposition \ref{coinduced-nij-tensor}, there exists a Nijenhuis tensor $\underline{N}$ on $M / \sim$ such that $q$ commute with the Nijenhuis tensors, that is, $q_* \circ N_M = \underline{N} \circ q_*$. Thus, it remains to show the compatibility of $\underline{\pi}$ and
$\underline{N}$ in the sense of Poisson-Nijenhuis structure. This follows from the argument similar to the last part of Proposition \ref{pn-grpd-base-pn}(ii).
\end{proof}

\section{P-N Lie bialgebroids}\label{sec6}
In this section, we introduce a special class of Lie bialgebroid, which we call P-N Lie bialgebroid and study some of its basic properties.
\begin{defn}\label{inf-mul-pn-lie-bialgbd}
 Let $A$ be a Lie algebroid over $M$. Suppose its dual bundle $A^*$ also carries a Lie algebroid structure with
$\pi_A$ being the induced linear Poisson structure on $A$. Then the pair $(A, A^*)$ is said to have a
{\sf P-N Lie bialgebroid} structure if
\begin{itemize}
 \item[(i)] the pair $(A, A^*)$ is a Lie bialgebroid over $M$,
 \item[(ii)] there exists an infinitesimal multiplicative Nijenhuis tensor $N_A : TA \rightarrow TA$ on $A$
such that $N_A^* : T^*A \rightarrow T^*A$ is a Lie algebroid morphism from the cotangent Lie algebroid $T^*A \rightarrow A^*$ to itself,
 \item[(iii)] $(\pi_A, N_A)$ is a Poisson-Nijenhuis structure on $A$.
\end{itemize}
\end{defn}
A P-N Lie bialgebroid over $M$ as above may be denoted by $(A, A^*, N_A).$
Note that the Poisson structure $\pi_A$ and the Nijenhuis tensor $N_A$ satisfy the following compatibility
conditions in the sense of Poisson-Nijenhuis structure, that is,
\begin{itemize}
 \item $N_A \circ \pi_A^\sharp =  \pi_A^\sharp \circ N_A^*$ \quad (that is, $N_A \circ \pi_A^\sharp$ defines a bivector field $N_A \pi_A$ on $A$),
 \item $C (\pi_A , N_A) \equiv 0,$
\end{itemize}
where
$$C(\pi_A, N_A) (\alpha, \beta) := [\alpha, \beta]_{N_A \pi_A} - ([N_A^*\alpha, \beta]_{\pi_A} + [\alpha, N_A^*\beta]_{\pi_A} - N_A^* [\alpha, \beta]_{\pi_A}),
\text{ for all } \alpha, \beta \in \Omega^1 (A).$$

\begin{exam}\label{exam-pn-lie-bialgbd}
(i) Any Lie bialgebroid $(A, A^*)$  can be considered as a P-N Lie bialgebroid  with $N_A = id$.

\medskip

(ii) (Poisson-Nijenhuis manifolds)

Let $(M, \pi, N)$ be a Poisson-Nijenhuis manifold. Then $(TM, (T^*M)_\pi)$ is a Lie bialgebroid, as $\pi$ being a Poisson structure.
Let $\widetilde{\pi}$ denote the dual linear Poisson structure on $TM$ induced from the cotangent Lie algebroid $(T^*M)_\pi$. Then $\widetilde{\pi} = \pi^c$
is the complete lift of $\pi$ to $TM$ \cite{gra-urb}. Moreover, the Nijenhuis tensor $N : TM \rightarrow TM$ induces an infinitesimal
multiplicative Nijenhuis tensor $\widetilde{N}= N_*: T(TM) \rightarrow T(TM)$ on the tangent Lie algebroid
\[
\xymatrixrowsep{0.4in}
\xymatrixcolsep{0.6in}
\xymatrix{
T(TM) \ar[r]^{N_*} \ar[d] & T(TM) \ar[d] \\
TM  \ar[r]_{N} & TM.
}
\]
The compatibility of $\widetilde{\pi}$ and $\widetilde{N}$ in the sense of Poisson-Nijenhuis structure follows from the compatibility of $\pi$ and $N$ \cite{gra-urb}.
Therefore, $(TM, (T^*M)_\pi, N_*)$ is a P-N Lie bialgebroid over $M$.

\medskip

(iii) (holomorphic Lie bialgebroids)

A {\sf holomorphic Lie algebroid} is a holomorphic vector bundle $A \rightarrow M$ whose sheaf of holomorphic sections $\mathcal{A}$ is a sheaf of complex Lie algebras
and there exists a holomorphic bundle map $\rho : A \rightarrow TM$, called the anchor, such that
\begin{itemize}
 \item $\rho$ induces a homomorphism of sheaves of complex Lie algebras from $\mathcal{A}$ to $\Theta_M$, where $\Theta_M$ is the sheaf of holomorphic vector fields on $M$;
 \item the Leibniz identity
$$ [X, fY] = f [X, Y] + (\rho(X)f ) Y$$
holds, for any open subset $U \subset M$ and $X, Y \in \mathcal{A}(U), f \in \mathcal{O}_M(U)$, where $\mathcal{O}_M$ is the sheaf of holomorphic functions on $M$ \cite{geng-stie-xu1,geng-stie-xu2}. 
\end{itemize}
A holomorphic Lie algebroid can equivalently described by a real Lie algebroid $A \rightarrow M$ which is also a holomorphic vector bundle such that the map
\[
\xymatrixrowsep{0.4in}
\xymatrixcolsep{0.6in}
\xymatrix{
TA \ar[r]^{J_A} \ar[d] & TA \ar[d] \\
TM  \ar[r]_{J_M} & TM
}
\]
defines a Lie algebroid morphism from the tangent Lie algebroid $TA \rightarrow TM$ to itself, where $J_A$ and $J_M$ denote the almost complex structures on $A$ and $M$, respectively \cite{geng-stie-xu1,geng-stie-xu2}.
A holomorphic Lie algebroid structure on $A$ also gives rise to a fibrewise linear holomorphic Poisson structure on $A^*_{\mathbb{C}}.$

A {\sf holomorphic Lie bialgebroid} is a pair of holomorphic Lie algebroids $(A, A^*)$ in duality over a base $M$ such that for any open subset $U \subset M$, the following
compatibility condition
$$d_* [X, Y] = [d_* X , Y ] + [X, d_* Y]$$
holds, for all $X, Y \in \mathcal{A}(U)$. Here $[~,~]$ denotes the sheaf of Gerstenhaber bracket on $(\mathcal{A}^\bullet, \wedge)$ induced from the holomorphic
Lie algebroid $A$ and $d_* : \mathcal{A}^k \rightarrow \mathcal{A}^{k+1}$ is the complex of sheaves over $M$ induced from the holomorphic Lie algebroid
$A^*$.

 It is known that if $(A, A^*)$ is a holomorphic Lie bialgebroid, then the underlying real Lie algebroids $(A, A^*)$ in duality forms a Lie bialgebroid
\cite{jotz-stienon-xu}.
Thus, it follows that if $(A, A^*)$ is a holomorphic Lie bialgebroid over $M$, then $(A, A^*, J_A)$ is a P-N Lie bialgebroid over $M$ (cf. Example \ref{pn-manifold-exam}(iii)).
\end{exam}

\begin{prop}\label{pn-lie-bialgbd-base-pn}
 Let $(A, A^*, N_A)$ be a P-N Lie bialgebroid over $M$. Then the base manifold $M$
carries a natural Poisson-Nijenhuis structure.
\end{prop}
\begin{proof}
Since the pair $(A, A^*)$ is a Lie bialgebroid over $M$,
the base $M$ carries a Poisson structure $\pi_M$ given by $\pi_M^\sharp = \rho_* \circ \rho^*$
\cite{mac-xu}. Here $\rho$ and $\rho_*$ being the anchor map of the Lie algebroid $A$ and $A^*$, respectively. Moreover, from condition (ii) of Definition \ref{inf-mul-pn-lie-bialgbd}, there exists a Nijenhuis tensor $N_M$ on $M$.

For any $\alpha \in T^*M$, we have

$$ N_M \circ \pi_M^\sharp (\alpha) = N_M \circ \rho_* (\rho^* \alpha) = (N_A \circ \pi_A^\sharp \circ R)\big|_0 (\rho^* \alpha)
= (\pi_A^\sharp \circ N_A^* \circ R)\big|_0 (\rho^* \alpha).$$
\[
\xymatrixrowsep{0.4in}
\xymatrixcolsep{0.6in}
\xymatrix{
                     & T^*A^* \ar[r]^{R} \ar[d] & T^*A \ar[r]^{\pi_A^\sharp} & TA \ar[d] \ar[r]^{N_A} &    TA \ar[d] \\
T^*M \ar[r]_{\rho^*} & A^*  \ar[rr]_{\rho_*}      &                            & TM \ar[r]_{N_M}        &    TM .
}
\]
Since $R : T^*A^* \rightarrow T^*A$ is the identity map when restricted to $A^* \hookrightarrow T^*A^*$ to $A^* \hookrightarrow T^*A$, we have
\begin{align*}
 N_M \circ \pi_M^\sharp (\alpha) = (\pi_A^\sharp \circ N_A^*)\big|_0 (\rho^* \alpha) &= (\pi_A^\sharp)\big|_0 \circ \rho^* (N_M ^* \alpha) \\
&= \rho_* \circ \rho^* (N_M ^* \alpha) = \pi_M^\sharp \circ N_M^* (\alpha).
\end{align*}
The Magri-Morosi concominant of $\pi_M$ and $N_M$ vanishes and it follows from the observation that any $1$-form $\alpha \in \Omega^1(M)$ can be considered as a $1$-form
$\widetilde{\alpha} \in \Omega^1(A)$ by
$$ \widetilde{\alpha} (X) = (\alpha (u), X, 0), \text{ for } X \in A_u, u \in M.$$
Since $(\pi_A, N_A)$ defines a Poisson-Nijenhuis structure on $A$, we have
$C(\pi_A, N_A) (\widetilde{\alpha}, \widetilde{\beta}) = 0$, for all $\alpha, \beta \in \Omega^1(M)$. As the restriction of $\pi_A$ and $N_A$ are respectively
$\pi_M$ and $N_M$, we have $C(\pi_M, N_M)(\alpha, \beta) = 0.$
\end{proof}

\begin{remark}
 Let $(M, \pi, N)$ be a Poisson-Nijenhuis manifold and consider the P-N Lie bialgebroid $(TM, (T^*M)_\pi, N_*)$ over $M$ (cf. Example \ref{exam-pn-lie-bialgbd}(ii)). The induced
Poisson structure on $M$ is given by $\pi$ and the induced Nijenhuis tensor is $N$. Therefore, the induced Poisson-Nijenhuis structure on $M$ coincides with the given one.
\end{remark}

\begin{prop}\label{pn-lie-bialgbd-comp-bialgbd}
Let $(A, A^*, N_A)$ be a  P-N Lie bialgebroid over $M$. Then there is a hierarchy of compatible Lie bialgebroid structures on $A$. 
\end{prop}

\begin{proof}
 Let $\pi_A$ being the linear Poisson structure on $A$ induced from the Lie algebroid $A^*$. Since $(A, \pi_A, N_A)$ forms a Poisson-Nijenhuis
manifold, there exists a hierarchy of compatible Poisson structures $N_A^k\pi_A$ on $A$, for $k \geq 0$. That is,
$$ [N_A^k \pi_A , N_A^k \pi_A] = 0  \quad \text{ ~~ and ~~ } \quad [N_A^k \pi_A , N_A^l \pi_A] = 0,$$
for $k, l \geq 0$. Note that the Lie algebroid morphism
\[
\xymatrixrowsep{0.4in}
\xymatrixcolsep{0.6in}
\xymatrix{
T^*A \oplus_A T^*A  \ar[d]  \ar[r]^{(id, N_A^*)}     & T^*A \oplus_A T^*A \ar[r]^{\pi_A} \ar[d]     & \mathbb{R}  \ar[d]\\
A^* \oplus_M A^* \ar[r]                              & A^* \oplus_M A^*  \ar[r]                     & {\ast}                  
}
\]
is given by $N_A \pi_A$. Thus, it follows that $N_A \pi_A$ defines a linear\footnote{Let $A \rightarrow M$ be a vector bundle. A $k$-vector field
$\Pi \in \Gamma (\largewedge^k TA)$ is called {\sf linear} if $\Pi$ takes $k$ fibrewise linear functions on $A$ to a fibrewise linear function on $A$. Equivalently,
if the map $\overline{\Pi} : \oplus_{A}^k  T^*A \rightarrow \mathbb{R},$ $\overline{\Pi} (\gamma_1, \ldots, \gamma_k) = \Pi (\gamma_1, \ldots, \gamma_k)$
is a vector bundle morphism from $\oplus_{A}^k  T^*A \rightarrow \oplus_M^k A^*$ to $\mathbb{R}$ \cite{burs-cab}.} bivector field on $A$. Similarly, for each $k \geq 0$,
the Poisson bivectors $N_A^k \pi_A$ on $A$ are linear.
Therefore, by Theorem \ref{comp-equiv}, they define compatible Lie algebroid structures on $A^*$.
Thus,
by the Remark \ref{comp-lie-bialgbd-rem}(i), there is a hierarchy of compatible Lie bialgebroid structures on $A$.
\end{proof}
\section{Infinitesimal form of Poisson-Nijenhuis groupoids}\label{sec7}

In this section, we show that the infinitesimal version of Poisson-Nijenhuis grouopids are P-N Lie bialgebroids.

\begin{prop}\label{pn-grpd-pn-lie-bialgbd}
 Let $(G \rightrightarrows M, \pi, N)$ be a Poisson-Nijenhuis groupoid with Lie algebroid $A$. Then
$(A, A^*, N_A)$ is a P-N Lie bialgebroid over $M$.
\end{prop}

\begin{proof}
 Since $(G \rightrightarrows M, \pi)$ is a Poisson groupoid with Lie algebroid $A$, the dual bundle $A^*$ also carries a Lie algebroid structure
and the pair $(A , A^*)$ forms a Lie bialgebroid \cite{mac-xu}. Let $\pi_A$ be the linear Poisson structure on $A$ induced from the Lie algebroid structure on $A^*$.
Then we have $$ \pi_A^\sharp = j_G^{-1} \circ Lie (\pi^\sharp) \circ j_G^{' -1} : T^*A \rightarrow TA.$$
Moreover, $N : TG \rightarrow TG$
is a multiplicative Nijenhuis tensor on $G$ implies that $ N_A := (\sigma^{-1} \circ TN \circ \sigma)|_A : TA \rightarrow TA $ is an infinitesimal multiplicative Nijenhuis tensor on $A$.
Therefore,
\begin{align*}
 N_A^* &= (\sigma^{-1} \circ TN \circ \sigma)^*\\
&= \sigma^* \circ Lie (N)^* \circ (\sigma^{-1})^* \\
&= j_G^{'} \circ i_G^{-1} \circ i_G \circ Lie (N^*) \circ i_G^{-1} \circ i_G \circ j_G^{' -1} \qquad (\text{since } j_G^{'} = \sigma^* \circ i_G)\\
&= j_G^{'} \circ Lie (N^*) \circ j_G^{' -1},
\end{align*}
where we have used the fact that $Lie (N)^* = i_G \circ Lie (N^*) \circ i_G^{-1}.$

Next, consider the Poisson groupoid $(G \rightrightarrows M, N\pi)$. Then the bundle $A^*$ carries a different Lie algebroid structure 
induced from the Poisson groupoid $(G \rightrightarrows M, N\pi)$. Let $\overline{\pi}_A$
be the linear Poisson structure on $A$ induced from this new Lie algebroid structure on $A^*$. We have
\begin{align*}
 \overline{\pi}_A^\sharp &= j_G^{-1} \circ Lie (N \circ \pi^\sharp) \circ j_G^{' -1} \\
&= j_G^{-1} \circ Lie(N) \circ j_G \circ j_G^{-1} \circ Lie (\pi^\sharp) \circ j_G^{' -1} = N_A \circ \pi_A^\sharp.
\end{align*}
On the other hand, since $N \circ \pi^\sharp = \pi^\sharp \circ N^*$, we also have
\begin{align*}
 \overline{\pi}_A^\sharp &= j_G^{-1} \circ Lie ( \pi^\sharp \circ N^*) \circ j_G^{' -1}\\
&= j_G^{-1} \circ Lie (\pi^\sharp) \circ j_G^{' -1} \circ j_G^{'} \circ Lie (N^*) \circ j_G^{' -1} = \pi_A^\sharp \circ N_A^*.
\end{align*}
Therefore, $N_A \circ \pi_A^\sharp = \pi_A^\sharp \circ N_A^*$ and hence, $N_A \circ \pi_A^\sharp$ defines a bivector field $N_A \pi_A$ on $A$.

To prove that the Magri-Morosi concominant of $\pi_A$ and $N_A$ vanishes, we need the following observation. Given a Poisson-Nijenhuis groupoid $(G \rightrightarrows M, \pi, N),$
the unit space $M \hookrightarrow G$ is a coisotropic-invariant submanifold of $G$ in the sense that $\pi^\sharp (TM)^0 \subset TM$ and $N (TM) \subset TM$.
Thus, it follows from an easy observation that $C(\pi, N) : \Gamma (T^*G) \times \Gamma(T^*G) \rightarrow \Gamma(T^*G)$ restricts to a map
from $\Gamma (TM)^0 \times \Gamma(TM)^0$ into $\Gamma(TM)^0$.

Consider the Lie groupoid $T^*G \oplus_G T^*G \rightrightarrows A^* \oplus_M A^*$, which is a subgroupoid of the direct product Lie groupoid
$T^*G \times T^*G \rightrightarrows A^* \times A^*$. The Lie algebroid of the groupoid $T^*G \oplus_G T^*G \rightrightarrows A^* \oplus_M A^*$ is isomorphic to
the Lie algebroid $T^*A \oplus_A T^*A \rightarrow A^* \oplus_M A^*$. Note that the Magri-Morosi concominant $C(\pi, N)$ is a Lie groupoid morphism from 
$T^*G \oplus_G T^*G \rightrightarrows A^* \oplus_M A^*$ to $T^*G \rightrightarrows A^*$, and the Magri-Morosi concominant $C(\pi_A, N_A)$ is then given by
$$ C(\pi_A, N_A) = j_G^{'} \circ Lie (C(\pi, N)) \circ (j_G^{'} \oplus j_G^{'})^{-1} : T^*A \oplus_A T^*A \rightarrow T^*A.$$
Since $C(\pi, N) \equiv 0$, it follows that $C(\pi_A, N_A) \equiv 0.$
\end{proof}

\begin{exam}

(i) If the Nijenhuis tensor $N = id$, then it follows that $N_A = id.$ Therefore, the infinitesimal form of Poisson groupoids are Lie bialgebroids.

\medskip

(ii) Let $(G \rightrightarrows M, \omega, N)$ be a symplectic-Nijenhuis groupoid. Since $(G \rightrightarrows M , \omega)$ is a symplectic groupoid, the base $M$
carries a Poisson structure $\pi_M$ such that the Lie algebroid $A$ of the Lie groupoid is isomorphic to the cotangent Lie algebroid $(T^*M)_{\pi_M}$ \cite{mac-xu2}.
The dual Lie algebroid $A^*$ is the usual tangent bundle Lie algebroid $TM$. 
Therefore, the dual linear Poisson structure on $A = T^*M$ is given by the canonical symplectic structure $\pi_{T^*M} = (\omega_{can})^{-1}$.
Moreover, since $N$ is a multiplicative Nijenhuis tensor on $G$ with base $N_M$, it induces an infinitesimal multiplicative Nijenhuis tensor $N_{T^*M}$ on $A = (T^*M)_{\pi_M}$.
\[
\xymatrixrowsep{0.4in}
\xymatrixcolsep{0.6in}
\xymatrix{
T^*(T^*M)  \ar[r]^{\pi_{T^*M}} \ar[d]     &  T(T^*M) \ar[r]^{N_{T^*M}} \ar[d]     &  T(T^*M) \ar[d] \\
TM  \ar[r]                              & TM  \ar[r]                     & TM               
}
\]
Then the Lie algebroid structure on $A^* = TM$ induced from the linear Poisson structure $N_{T^*M} \pi_{T^*M}$ on $A=T^*M$ is given by $(TM)_{N_M}$. That is, the
anchor for this Lie algebroid is $N_M$ and the differential $d_{N_M} : \Gamma(\largewedge^\bullet T^*M) \rightarrow \Gamma(\largewedge^{\bullet +1} T^*M)$ is
$$ d_{N_M} = i_{N_M} \circ d - d \circ i_{N_M},$$
where $i_{N_M}$ denotes the derivation of degree zero defined by
$$ (i_{N_M} \alpha)(X_1, \ldots, X_k) = \sum_{i=1}^k \alpha (X_1, \ldots, N_M X_i, \ldots, X_k),$$
for $\alpha \in \Gamma(\largewedge^k T^*M)$, $X_1, \ldots, X_k \in \Gamma(TM)$ \cite{yks2}. 
This Lie algebroid structure $(TM)_{N_M}$ on $A^*$ is also given by the Poisson grouopid $(G \rightrightarrows M, N \omega^{-1})$. Therefore,
 $((T^*M)_{\pi_M}, d_{N_M})$ is a Lie bialgebroid. By
Proposition \ref{lie-bialgbd-pn}, this Lie bialgebroid is induced by the Poisson-Nijenhuis structure $(\pi_M, N_M)$.

\medskip

(iii) Let $(M , \pi, N)$ be a Poisson-Nijenhuis manifold and consider the Poisson-Nijenhuis groupoid
$(M \times M \rightrightarrows M, \pi \ominus \pi, N \oplus N)$ given by Example \ref{pn-grpd-exam}(iii). The Lie algebroid $A$ of the pair groupoid is the
usual tangent Lie algebroid $TM$ of $M$. The dual Lie algebroid $A^*$ is the cotangent Lie algebroid $(T^*M)_\pi$. Moreover, the multiplicative Nijenhuis tensor
$N \oplus N$ on the pair groupoid induces $\widetilde{N} = N_*$ as the infinitesimal multiplicative Nijenhuis tensor on the Lie algebroid $A = TM$. 
Therefore, the P-N Lie bialgebroid associated to the Poisson-Nijenhuis groupoid $(M \times M \rightrightarrows M, \pi \ominus \pi, N \oplus N )$ is given by
$(TM, (T^*M)_\pi, N_*)$ considered in Example \ref{exam-pn-lie-bialgbd}(ii).
\end{exam}

\begin{remark}
 Thus, for a given Poisson-Nijenhuis groupoid $(G \rightrightarrows M, \pi, N)$ with Lie algebroid $A$, there is a hierarchy of compatible Lie
bialgebroid structures on $A$. The hierarchy of Lie bialgebroid structures on $A$ induced from the Poisson-Nijenhuis groupoid $(G \rightrightarrows M, \pi, N)$
and from the P-N Lie bialgebroid $(A, A^*, N_A)$ are same (cf. Remark \ref{pn-grpd-comp-liebialgbd}, Proposition \ref{pn-lie-bialgbd-comp-bialgbd}).
\end{remark}

\begin{thm}\label{1-1,pn-grpd-pn-lie-bialgbd}
 Let $G \rightrightarrows M$ be a $s$-connected and $s$-simply connected Lie groupoid with Lie algebroid $A$. Then
there is a one-to-one correspondence between multiplicative Poisson-Nijenhuis structures on $G$ and P-N Lie bialgebroid structures on $A$.
\end{thm}

\begin{proof}
 Let $(A, A^*, N_A)$ be a P-N Lie bialgebroid structure on $A$. Since $(A, A^*)$ is a Lie bialgebroid, there is a Poisson structure
$\pi$ on $G$ which makes $(G \rightrightarrows M, \pi)$ into a Poisson groupoid with Lie bialgebroid $(A, A^*)$. Moreover, $N_A :  TA \rightarrow TA$ is an
infinitesimal multiplicative Nijenhuis tensor on $A$. Therefore, by Theorem \ref{mul-inf-mul}, there is a multiplicative Nijenhuis tensor
$N : TG \rightarrow TG$ such that
$$ N_A = (\sigma^{-1} \circ TN \circ \sigma)\big|_A.$$
It remains to show that $\pi$ and $N$ are compatible in the sense of Poisson-Nijenhuis structure. 
We have 
$$N_A \circ \pi_A^\sharp = j_G^{-1} \circ Lie(N) \circ j_G \circ j_G^{-1} \circ Lie (\pi^\sharp) \circ j_G^{' -1} = j_G^{-1} \circ Lie (N \circ \pi^\sharp) \circ j_G^{' -1} $$
and
$$ \pi_A^\sharp \circ N_A^* = j_G^{-1} \circ Lie(\pi^\sharp) \circ j_G^{' -1} \circ j_G^{'} \circ Lie (N^*) \circ j_G^{' -1} =  j_G^{-1} \circ Lie (\pi^\sharp \circ N^*) \circ j_G^{' -1}.$$
Since $N_A \circ \pi_A^\sharp = \pi_A^\sharp \circ N_A^*$ and $j_G, j_G^{'}$ are isomorphisms, we have $Lie (N \circ \pi^\sharp) = Lie (\pi^\sharp \circ N^*).$
As the Lie groupoid $T^*G \rightrightarrows A^*$ is  $s$-connected, $s$-simply connected and both
$N \circ \pi^\sharp$, $\pi^\sharp \circ N^* : T^*G \rightarrow TG$ are Lie groupoid morphisms from $T^*G \rightrightarrows A^*$ to
$TG \rightrightarrows TM$, it follows that $N \circ \pi^\sharp = \pi^\sharp \circ N^*$.
%
%
Therefore,
$N \pi$ defines a bivector field on $G$. Finally, the same argument as in Proposition \ref{pn-grpd-pn-lie-bialgbd} shows that 
$ 0 = C(\pi_A, N_A) = j_G^{'} \circ Lie (C(\pi, N)) \circ (j_G^{'} \oplus j_G^{'})^{-1}$. Since 
$T^*G \oplus_G T^*G \rightrightarrows A^* \oplus_M A^*$ is $s$-connected and $s$-simply connected Lie groupoid, we have $C(\pi, N) = 0$.
\end{proof}

\section{Coisotropic-invariant subgroupoids}\label{sec8}
 In this section, we introduce a class of subgroupoids of a Poisson-Nijenhuis groupoid generalizing coisotropic subgroupoids of a Poisson grouopid \cite{xu}.

\begin{defn}
 Let $(M, \pi, N)$ be a Poisson-Nijenhuis manifold. A submanifold $S \hookrightarrow M$ is called {\sf coisotropic-invariant} if $S$ is coisotropic with respect
to $\pi$ and invariant with respect to $N$, that is,
$$ \pi^\sharp(TS)^0 \subset TS ~~~ \text{~ and ~} ~~~  N (TS) \subset TS .$$
\end{defn}

Thus, it follows that if $S \hookrightarrow M$ is a coisotropic-invariant submanifold of $(M, \pi, N)$, then for each $k \geq 0$, the submanifold $S$ is coisotropic
with respect to the Poisson tensor $\pi_k = N^k \pi$ on $M$.

\begin{exam}
 Let $(G \rightrightarrows M, \pi, N)$ be a Poisson-Nijenhuis groupoid. Then the unit space $M \hookrightarrow G$ is a coisotropic-invariant submanifold of $(G, \pi, N)$.
\end{exam}
In the following proposition, we characterize Poisson-Nijenhuis maps (P-N maps) in terms of coisotropic-invariant submanifolds of the product manifold.
The result follows from the result of Weinstein \cite{wein} which characterize Poisson maps in terms of coisotropic submanifolds of the product manifold and
our Proposition \ref{nij-map-grpah-nij-submnfld}.
\begin{prop}
 Let $(M_1, \pi_1, N_1)$ and $(M_2, \pi_2, N_2)$ be two Poisson-Nijenhuis manifolds and $\phi : M_1 \rightarrow M_2$ be a smooth map. Then $\phi$ is a P-N map
if and only if
$$Gr(\phi) = \{(x, \phi(x)) |~ x \in M_1 \}$$
is a coisotropic-invariant submanifold of the Poisson-Nijenhuis manifold $(M_1 \times M_2, \pi_1 \ominus \pi_2, N_1 \oplus N_2)$.
\end{prop}

\begin{defn}
 Let $(G \rightrightarrows M, \pi, N)$ be a Poisson-Nijenhuis groupoid. A subgroupoid $H \rightrightarrows S$ is called a {\sf coisotropic-invariant subgroupoid}
if $H$ is a coisotropic-invariant submanifold of $(G, \pi, N)$.
\end{defn}

\begin{exam}
(i) Let $(G \rightrightarrows M, \pi)$ be a Poisson groupoid considered as a Poisson-Nijenhuis groupoid with $N = id$. Then coisotropic subgroupoids of the Poisson groupoid $(G \rightrightarrows M, \pi)$ are same as coisotropic-invariant
subgroupoids of the Poisson-Nijenhuis groupoid $(G \rightrightarrows M, \pi, N = id) .$

\medskip

(ii) Let $(G \rightrightarrows M, \pi, N)$ be a Poisson-Nijenhuis groupoid. Then from Proposition \ref{pn-grpd-base-pn}, the base $M$ carries a Poisson-Nijenhuis structure
$(\pi_M, N_M)$ such that the source $s$ is a P-N map and the target $t$ is an anti P-N map. Let $S$ be a coisotropic-invariant submanifold of $(M , \pi_M, N_M)$. Then the restriction $G|_S := s^{-1}(S) \cap t^{-1}(S) \rightrightarrows S$
is a coisotropic-invariant subgroupoid of $(G \rightrightarrows M, \pi, N)$.
\end{exam}

To study the infinitesimal object corresponding to coisotropic-invariant subgroupoids of a Poisson-Nijenhuis grouopid, we introduce coisotropic-invariant
subalgebroids of a P-N Lie bialgebroid.

\begin{defn}\label{coiso-inv-subalgbd}
 Let $(A, A^*, N_A)$ be a P-N Lie bialgebroid over $M$. Then a Lie subalgebroid $B \rightarrow S$ of $A \rightarrow M$ is called a {\sf coisotropic-invariant subalgebroid}
of $(A, A^*, N_A)$ if $B \hookrightarrow A$ is a coisotropic-invariant submanifold with respect to the Poisson-Nijenhuis structure $(\pi_A, N_A)$.
\end{defn}

Thus, it follows from Definition \ref{coiso-inv-subalgbd} that a Lie subalgebroid $B \rightarrow S$ is a coisotropic-invariant subalgebroid of $(A, A^*, N_A)$
if and only if $B \rightarrow S$ is a coisotropic subalgebroid of the Lie bialgebroid $(A, A^*)$ and the inclusion $B \hookrightarrow A$ is moreover an invariant submanifold
with respect to the tensor $N_A$.

\begin{exam}
Let $S$ be a coisotropic-invariant submanifold of a Poisson-Nijenhuis manifold $(M, \pi, N)$.
Since $S$ is a coisotropic submanifold of $(M, \pi)$, it follows that $TS \rightarrow S$ is a coisotropic subalgebroid of the Lie bialgebroid $(TM, (T^*M)_\pi)$ \cite{xu}.
Moreover, $S \hookrightarrow M$ is an invariant submanifold with respect to $N$ implies that $TS \hookrightarrow TM$ is an invariant submanifold with respect
to $N_*$. This shows that $TS \rightarrow S$ is a coisotropic-invariant subalgebroid of the P-N Lie bialgebroid $(TM, (T^*M)_\pi, N_*)$. 
\end{exam}

It is known that the base of a P-N Lie bialgebroid carries a Poisson-Nijenhuis structure (cf. Proposition \ref{pn-lie-bialgbd-base-pn}). The next proposition shows that the base of a coisotropic-invariant subalgebroid is
coisotropic-invariant with respect to the induced Poisson-Nijenhuis structure.

\begin{prop}
 Let $(A, A^*, N_A)$ be a P-N Lie bialgebroid over $M$. If $B \rightarrow S$ is a coisotropic-invariant subalgebroid of $(A, A^*, N_A)$, then $S$ is a coisotropic-invariant submanifold
of $M$.
\end{prop}
\begin{proof}
 Since $B \rightarrow S$ is a coisotropic subalgebroid of the Lie bialgebroid $(A, A^*),$ it follows that $S$ is a coisotropic submanifold of $(M, \pi_M)$ \cite{xu}. 
Moreover, $B \rightarrow S$ is a Lie subalgebroid of $A \rightarrow M$ implies that $TB \rightarrow TS$ is a Lie subalgebroid of $TA \rightarrow TM$.
Since $B \hookrightarrow A$ is an invariant submanifold with respect to $N_A$, it follows that $S \hookrightarrow M$ is an invariant submanifold with respect to the tensor
$N_M$. Thus, $S$ is a coisotropic-invariant submanifold
of $(M, \pi_M, N_M).$
\end{proof}

\begin{prop}
 Let $(G \rightrightarrows M, \pi, N)$ be a Poisson-Nijenhuis groupoid with P-N Lie bialgebroid $(A, A^*, N_A)$. Suppose that $H \rightrightarrows S$ is a coisotropic-invariant
subgroupoid of $G \rightrightarrows M$ with Lie algebroid $B \rightarrow S$, then $B \rightarrow S$ is a coisotropic-invariant subalgebroid $(A, A^*, N_A)$.
\end{prop}

\begin{proof}
 Since $H \rightrightarrows S$ is a coisotropic subgroupoid of the Poisson groupoid $(G \rightrightarrows M, \pi),$ its Lie algebroid $B \rightarrow S$
is a Lie subalgebroid of the Lie bialgebroid $(A, A^*)$ \cite{xu}. Moreover, we have
$$ B_q = A_q \cap T_qH, ~~ \text{ for } q \in S.$$
Therefore, it follows that $TB = TA \cap T(TH)$, which implies
$$ N_A (TB) \subseteq ~ N_A (TA) \cap \text{Lie}(N) (T(TH)) \subseteq ~  TA \cap T(TH) = TB .$$
Thus, $B \hookrightarrow A$ is an invariant submanifold with respect to the tensor $N_A$.
\end{proof}

\begin{corollary}
 Let $(G \rightrightarrows M, \pi, N)$ be a Poisson-Nijenhuis groupoid with induced Poisson-Nijenhuis structure $(\pi_M, N_M)$ on $M$. If $H \rightrightarrows S$ is a coisotropic-invariant subgroupoid, then $S \hookrightarrow M$ is a coisotropic-invariant
submanifold with respect to $(\pi_M, N_M)$.
\end{corollary}

\section{Poisson-Nijenhuis actions}\label{sec9}
In this section, we define Poisson-Nijenhuis action (or P-N action in short) as a generalization of Poisson action introduced by Liu, Weinstein and Xu \cite{liu-wein-xu}.

\begin{defn}\label{defn-pn-action}
 Let $(G \rightrightarrows M, \pi, N)$ be a Poisson-Nijenhuis groupoid. A {\sf P-N action} of $G$ on a Poisson-Nijenhuis manifold $(X, \pi_X, N_X)$ is a Lie groupoid
action of $G$ on $X$ with moment map $J : X \rightarrow M$ such that
\begin{itemize}
 \item[(i)] the action is Poisson \cite{liu-wein-xu}, that is, the graph 
$$ \Omega = \{(g, x, g x) |~ t (g) = J(x)\} \subset G \times X \times \overline{X}$$
of the action is a coisotropic submanifold,
 \item[(ii)] the graph $\Omega$ is an invariant submanifold of $G \times X \times X$ with respect to the tensor $N \oplus N_X \oplus N_X.$
\end{itemize}
\end{defn}

Thus, a P-N action of a Poisson-Nijenhuis groupoid $(G \rightrightarrows M, \pi, N)$ on a Poisson-Nijenhuis manifold $(X, \pi_X, N_X)$ is a groupoid action with moment map
$J : X \rightarrow M$ such that the graph of the action $\Omega$ is a coisotropic-invariant submanifold of the Poisson-Nijenhuis manifold
$(G \times G \times G, \pi \oplus \pi_X \ominus \pi_X, N \oplus N_X \oplus N_X)$.

If $\pi$ and $\pi_X$ are non-degenerate Poisson structures (that is, symplectic structures) on $G$ and $X$, respectively, then the P-N action is called symplectic-Nijenhuis action.

\begin{exam}
 Let $G \rightrightarrows M$ be a Lie groupoid. Then there is a Lie groupoid action of $G$ on itself with the source $s : G \rightarrow M$
as the moment map and the groupoid multiplication as action map. In particaular, if $(G \rightrightarrows M, \pi, N)$ is a Poisson-Nijenhuis groupoid, the
action of $G$ on itself is a P-N action.
\end{exam}

Let $(G \rightrightarrows M, \pi, N)$ be a Poisson-Nijenhuis groupoid with a P-N action on Poisson-Nijenhuis manifold $(X, \pi_X, N_X)$ with moment map
$J: X \rightarrow M$. Then for each $k \geq 0$ and for any $(\alpha_g, \beta_x, \gamma_{gx}) \in (T_{(g,x,gx)} \Omega)^0$, we have
$$ (N^k \pi \oplus N^k_X \pi_X \oplus \overline{N^k_X \pi_X})(\alpha_g, \beta_x, \gamma_{gx}) = (N \oplus N_X \oplus N_X)^k \big( \pi^\sharp (\alpha_g), \pi_X^\sharp (\beta_x), - \pi_X^\sharp(\gamma_{gx})   \big)$$
is in $T_{(g,x,gx)} \Omega$. Thus, for each $k \geq 0,$ the action is a Poisson action of the Poisson groupoid
$(G \rightrightarrows M, N^k \pi)$ on the Poisson manifold $(X, N^k_X \pi_X)$.

\begin{prop}
 Let $(G \rightrightarrows M, \pi, N)$ be a Poisson-Nijenhuis groupoid. If $G$ has a P-N action on a Poisson-Nijenhuis manifold $(X, \pi_X, N_X)$ wih moment map
$J : X \rightarrow M$, then $J$ is a P-N map.
\end{prop}

\begin{proof}
 Since the action is Poisson, it follows that the moment map $J : X \rightarrow M$ is a Poisson map \cite{liu-wein-xu}. Thus, it remains to show that $J$
is a Nijenhuis map.

Take $x \in X$ and $\delta_x \in T_xX$. Consider the vector $(\epsilon_* {J_* \delta_x}, \delta_x , \delta_x)$ tangent to $\Omega$ at the point $(\epsilon{(J(x))}, x, x) \in \Omega$.
Since $\Omega$ is an invariant submanifold of $G \times X \times X$ with respect to the tensor $N \oplus N_X \oplus N_X$, we have
$$ N( \epsilon_* {(J_* \delta_x)}) = \epsilon_* {(J_* N_X(\delta_x))}.$$
It follows that, $J_* N_X(\delta_x) = N_M (J_* \delta_x)$, proving that $J$ is a Nijenhuis map.
\end{proof}

\begin{remark}
 In \cite{he-liu-zhong}, the authors characterize Poisson actions of Poisson groupoid $G$ on Poisson manifold $X$ by Lie bialgebroid morphisms from
$T^*X$ to $A^*$, where $A$ is the Lie algebroid of the groupoid $G$. Thus, it is interesting to study P-N Lie bialgebroid morphism and characterize P-N actions 
in terms of P-N Lie bialgebroid morphisms.
\end{remark}


%

\end{document}